\documentclass[a4paper,11pt,reqno]{amsart}
\usepackage[T1]{fontenc}
\usepackage{lmodern}
\usepackage{amsmath,amssymb,amsthm,mathtools}
\usepackage[margin=27mm]{geometry}
\usepackage{enumitem,needspace,placeins}
\usepackage{xcolor}
\usepackage[mathlines]{lineno}

\usepackage{tikz}
\usetikzlibrary{fit,shapes.geometric}
\usepackage[colorlinks=true,linkcolor=black,citecolor=black,urlcolor=black]{hyperref}
\usepackage{xurl}
\hypersetup{pdftitle={Ehrhart positivity of series-parallel matroids},
 pdfauthor={},pdfsubject={Positive expansions for load polynomials of S-P trees}}

\theoremstyle{plain}
\newtheorem{theorem}{Theorem}[section]
\newtheorem{lemma}[theorem]{Lemma}
\newtheorem{corollary}[theorem]{Corollary}
\newtheorem{conjecture}[theorem]{Conjecture}
\newtheorem{proposition}[theorem]{Proposition}
\theoremstyle{definition}
\newtheorem{definition}[theorem]{Definition}
\newtheorem{example}{Example}[section]
\theoremstyle{remark}
\newtheorem{remark}[theorem]{Remark}
\theoremstyle{plain}
\numberwithin{equation}{section}
\allowdisplaybreaks[2]
\setlist{itemsep=0.2em,topsep=0.4em}

\title[Series-parallel extensions and Ehrhart positivity]{Series-parallel extensions and Ehrhart positivity}
\author{  YAO LI AND SIHONG SHAO}
\date{}

\begin{document}
\begin{abstract}
We prove Ehrhart positivity for series-parallel matroids and extend this
result to all series-parallel extensions of multi-path matroids.
A formula for the Ehrhart coefficients of series-parallel matroids is
derived and used to prove strict positivity.
An extension criterion based on weighted lattice-point counts is established.
Recursive formulas for lattice-point counts are derived, and the recursion is
shown to preserve a class of polynomials with nonnegative coefficients.
We finally construct a connected transversal matroid with a
negative Ehrhart coefficient, disproving the conjecture of Ferroni,
Morales, and Panova that every connected matroid whose base polytope
admits a subdivision into series-parallel matroid base polytopes is
Ehrhart positive.
\end{abstract}
\maketitle

\section{Introduction}\label{sec:introduction}\label{ux9605ux8bfbux7ea6ux5b9aux4e0eux8bc1ux660eux8303ux56f4}

Let \(E\) be a finite set. For a polytope
\(\mathcal P\subseteq\mathbb R^E\) with vertices in \(\mathbb Z^E\), define
\[
i(\mathcal P,t)=\#\bigl(t\mathcal P\cap\mathbb Z^E\bigr).
\]
Unless otherwise stated, the dilation parameter \(t\) is a nonnegative integer.
Ehrhart's theorem states that \(i(\mathcal P,t)\) is a polynomial
in \(t\) with rational coefficients and degree \(\dim\mathcal P\);
see \cite[Theorem~3.8]{BeckRobinsComputing}. It is called the
\emph{Ehrhart polynomial} of \(\mathcal P\), and its constant term
is \(1\). We say that \(\mathcal P\) is \emph{Ehrhart positive} if
all coefficients of \(i(\mathcal P,t)\) are nonnegative.
A particularly rich source of lattice polytopes is provided by matroids.
Matroid base polytopes are generalized permutohedra
\cite{PostnikovPermutohedra,ArdilaBenedettiDokerMatroidPolytopes}.
For a matroid \(M\) on \(E\), let \(\mathcal B(M)\) denote its set of bases.
Its base polytope is
\begin{equation}
\mathcal P(M)=\operatorname{conv}\{\mathbf1_B:B\in\mathcal B(M)\},
\label{eq:matroid-base-polytope}
\end{equation}
where \(\mathbf1_B\in\{0,1\}^E\) is the indicator vector of \(B\).
We write \(i(M,t):=i(\mathcal P(M),t)\), and call \(M\)
\emph{Ehrhart positive} when \(\mathcal P(M)\) is Ehrhart positive.

De Loera, Haws, and K\"oppe studied Ehrhart polynomials of matroid
polytopes and conjectured that every matroid is Ehrhart positive
\cite[Conjecture~2(B)]{DeLoeraHawsKoppeEhrhart}.
Castillo and Liu formulated the broader conjecture that all integral
generalized permutohedra are Ehrhart positive
\cite{CastilloLiuPermutohedra}.
Ferroni disproved both conjectures by constructing matroids with
negative Ehrhart coefficients \cite{FerroniNotEhrhartPositive}.
Ehrhart positivity has nevertheless been established for several
classes of matroids. Ferroni proved positivity for hypersimplices,
equivalently for uniform matroids \cite{FerroniHypersimplices},
and later for minimal matroids \cite{FerroniMinimal}.
Ferroni, Jochemko, and Schr\"oter proved positivity for all
matroids of rank two \cite{FerroniJochemkoSchroterRankTwo}.
Hanely et al.\ derived Ehrhart polynomial formulas for paving and
panhandle matroids \cite{HanelyPavingPanhandle}.
Deligeorgaki, McGinnis, and Vindas-Mel\'endez subsequently proved
positivity for all panhandle matroids
\cite{DeligeorgakiMcGinnisVindasPanhandle}.
Fan and Li proved positivity for several families of Schubert matroids
\cite{FanLiSchubert}, and Chen, Li, and Yao proved Ehrhart positivity for Catalan matroids
\cite{ChenLiYaoCatalan}. Ferroni, Morales, and Panova proved positivity
for snake matroids \cite{FerroniMoralesPanovaSkew} and, more recently,
for all lattice path matroids
\cite[Theorem~1.2]{FerroniMoralesPanovaLPM}.

The totally nonnegative Grassmannian \(\operatorname{Gr}_{\ge0}(k,n)\)
consists of the \(k\)-dimensional subspaces of \(\mathbb R^n\) whose
Pl\"ucker coordinates can be chosen nonnegative.
Postnikov decomposed this space into cells by prescribing which
Pl\"ucker coordinates vanish; the matroids indexing these cells are
called \emph{positroids} \cite[Section~3]{PostnikovPositiveGrassmannian}.
Ferroni, Jochemko, and Schr\"oter conjectured that all positroids
are Ehrhart positive \cite[Conjecture~6.3]{FerroniJochemkoSchroterRankTwo}.
Both lattice path
matroids and series-parallel matroids belong to this class, and
their common connected members with at least two elements are
precisely the snake matroids
\cite[Section~1.1]{FerroniMoralesPanovaLPM}.
A connected series-parallel matroid with at least two elements is
obtained from \(U_{1,2}\), the rank-one uniform matroid on two elements,
by a sequence of series and parallel extensions.

A connected matroid with at least two elements is series-parallel if
and only if its Crapo beta invariant equals one
\cite[Proposition~2.1]{FerroniLarsonBraid}.
Series-parallel matroids play a structural role in the geometry of
positroids. Speyer and Williams proved that every maximal cell in a
finest regular subdivision of a hypersimplex into positroid base
polytopes is a series-parallel matroid base polytope
\cite[Theorem~6.6]{SpeyerWilliamsPositiveDressian}.
Ferroni, Morales, and Panova posed Ehrhart positivity for
series-parallel matroids in \cite{FerroniMoralesPanovaSkew}.
We resolve their question.

\begin{theorem}\label{thm:main-introduction}
Series-parallel matroids are Ehrhart positive.
\end{theorem}

For a connected series-parallel matroid \(M\) with at least two
elements, Corollary~\ref{prop:marked-cycles} expresses each Ehrhart
coefficient as a sum of nonnegative rational weights over permutations
with marked cycles. A positive summand occurs in every degree
\(0\le h\le\dim\mathcal P(M)\), yielding strict positivity
(Corollary~\ref{thm:strict-root}).

We consider lattice-point counts weighted by products of binomial
polynomials. For a matroid \(M\) on \([n]\) and two vectors
\(\mathbf c,\mathbf d\in\mathbb Z_{\ge0}^n\), define
\begin{equation}
W_{M;\mathbf c,\mathbf d}(t)
=\sum_{x\in t\mathcal P(M)\cap\mathbb Z^n}
\prod_{i=1}^n\binom{x_i+c_i}{c_i}
              \binom{t-x_i+d_i}{d_i},
\qquad t\in\mathbb Z_{\ge0}.
\label{eq:atomic-weighted-count}
\end{equation}
The weighted count \(W_{M;\mathbf c,\mathbf d}(t)\) is a polynomial
by weighted Ehrhart theory; see Section~\ref{subsec:tree-substitution}.

\begin{lemma}\label{lem:extension-criterion}
Let \(M\) be a connected matroid on \([n]\) with \(n\ge2\).
Assume that \(W_{M;\mathbf c,\mathbf d}\) has nonnegative coefficients
for all \(\mathbf c,\mathbf d\in\mathbb Z_{\ge0}^n\).
Then every matroid obtained from \(M\) by
finitely many series and parallel extensions, including zero
extensions, is Ehrhart positive.
\end{lemma}

We use Lemma~\ref{lem:extension-criterion} to establish Ehrhart
positivity for a larger class of matroids. Multi-path matroids,
introduced by Bonin and Gim\'enez \cite{BoninGimenezMultiPath},
generalize lattice path matroids by allowing multiple admissible
pairs of starting and ending points between fixed boundaries;
see Section~\ref{subsec:multipath-background}.
Let \(\mathrm{MPM}\) denote the class of multi-path matroids and
\(\mathrm{SPEX}(\mathrm{MPM})\) the class obtained from them by
finitely many series and parallel extensions, including zero
extensions and relabeling. Theorem~\ref{thm:weighted-mpm} verifies
the hypothesis of Lemma~\ref{lem:extension-criterion} for every
connected multi-path matroid with at least two elements, yielding
the following result.

\begin{theorem}\label{cor:spex-mpm}\label{thm:spex-mpm}
Every matroid in \(\mathrm{SPEX}(\mathrm{MPM})\) is Ehrhart positive.
\end{theorem}

\begin{remark}\label{rem:main-special-cases}
\begin{enumerate}[label=\textup{(\roman*)}]
\item Theorems~\ref{thm:main-introduction} and~\ref{thm:spex-mpm}
recover Ehrhart positivity for snake matroids and lattice path
matroids, respectively. Our proofs differ from the proofs using
order polynomials of fence posets in
\cite[Corollary~5.6]{FerroniMoralesPanovaSkew} and
\cite[Theorem~1.2]{FerroniMoralesPanovaLPM}.
\item Multi-path matroids are positroids
\cite[Section~2.6]{BoninCharacterizationPositroids}, and positroids are closed
under series and parallel extensions
\cite[Corollary~4.18]{BoninCharacterizationPositroids}.
Thus \(\mathrm{SPEX}(\mathrm{MPM})\) consists of positroids, so
Theorem~\ref{thm:spex-mpm} supports the conjecture that all
positroids are Ehrhart positive
\cite[Conjecture~6.3]{FerroniJochemkoSchroterRankTwo}.
\end{enumerate}
\end{remark}

Multi-path matroids belong to the class of \emph{transversal matroids}
\cite[Section~2]{BoninGimenezMultiPath}.
A transversal matroid can be represented by a bipartite graph:
its ground set is one vertex class, and a subset is independent
precisely when its vertices can be matched to distinct vertices
in the other class.

The base polytopes of positroids and transversal matroids admit
subdivisions into series-parallel matroid base polytopes
\cite[Section~4.4]{FerroniMoralesPanovaLPM}.
Proposition~\ref{prop:transversal-counterexample} constructs a
connected transversal matroid with a negative Ehrhart coefficient.
This disproves the conjecture of Ferroni, Morales, and Panova
\cite[Conjecture~4.7]{FerroniMoralesPanovaLPM}.
Figure~\ref{fig:class-inclusions} summarizes the classes considered
in this paper.

\begin{figure}[!ht]
\centering
\begin{tikzpicture}[
  family/.style={draw,rectangle,align=center,font=\small,
    minimum height=9mm,inner xsep=7pt,inner ysep=5pt},
  established/.style={family,ellipse},
  conjectural/.style={family,diamond,aspect=2.2,inner xsep=3pt,inner ysep=3pt},
  positive/.style={family,rectangle,line width=1pt},
  counterexample/.style={family,rectangle,densely dotted,line width=0.8pt},
  inclusion/.style={->,line width=0.6pt,>=stealth},
  every node/.style={font=\small}
]
\node[font=\small\itshape] at (-0.15,9.5) {Matching polytopes};
\node[font=\small\itshape] at (7.1,9.5) {Matroids};
\node[counterexample,text width=71mm] (subdiv) at (7.1,8)
  {Matroids that admit subdivisions\\into series-parallel matroids\\
   {\scriptsize\bfseries Not Ehrhart positive}\\
   {\scriptsize Proposition~\ref{prop:transversal-counterexample}}};
\node[conjectural,text width=19mm] (pos) at (4.6,5.6)
  {Positroids};
\node[counterexample,text width=38mm] (trans) at (9.6,5.6)
  {Transversal matroids\\
   {\scriptsize\bfseries Not Ehrhart positive}\\
   {\scriptsize Proposition~\ref{prop:transversal-counterexample}}};
\node[positive,text width=30mm] (spex) at (4.6,3.8)
  {\(\mathrm{SPEX}(\mathrm{MPM})\)\\
   \scriptsize Theorem~\ref{thm:spex-mpm}};
\node[positive,text width=31mm] (mpm) at (9.6,3.35)
  {Multi-path matroids};
\node[positive,text width=30mm] (sp) at (4.6,1.65)
  {Series-parallel matroids\\
   \scriptsize Theorem~\ref{thm:main-introduction}};
\node[established,text width=28mm] (lpm) at (9.6,1.65)
  {Lattice path matroids};
\node[established,text width=25mm] (snake) at (7.1,0)
  {Snake matroids};
\node[conjectural,aspect=1.4,minimum width=35mm,minimum height=25mm,
  inner sep=0pt] (matching) at (-0.15,5.6) {};
\node[align=center,font=\footnotesize,inner sep=0pt,yshift=0.5mm] at (matching.center)
  {Matching\\[-2pt]polytopes of finite\\[-2pt]simple graphs\\[-2pt]
   \scriptsize Conjecture~\ref{conj:matching-ehrhart-positive}};
\node[positive,text width=35mm] (forest) at (-0.15,1.65)
  {Forest matching\\polytopes\\
   \scriptsize Corollary~\ref{cor:forest-matching-positive}};
\draw[inclusion] (snake) -- (sp);
\draw[inclusion] (snake) -- (lpm);
\draw[inclusion] (sp) -- (spex);
\draw[inclusion] (lpm) -- (mpm);
\draw[inclusion] (mpm) -- (spex);
\draw[inclusion] (spex) -- (pos);
\draw[inclusion] (mpm) -- (trans);
\draw[inclusion] (pos) -- (subdiv);
\draw[inclusion] (trans) -- (subdiv);
\draw[inclusion] (forest) -- (matching);
\draw[inclusion,dashed] (forest) -- (sp);
\end{tikzpicture}
\caption{Solid arrows denote inclusions, with matroids considered up to
isomorphism. The dashed arrow means that each forest matching polytope
is integrally affinely isomorphic to the base polytope of a
series-parallel matroid; see Section~\ref{sec:matching-polytopes}.
Ellipses denote classes whose Ehrhart positivity was previously known;
diamonds denote classes whose Ehrhart positivity remains conjectural.
Rectangles denote results of this paper: thick solid borders indicate
Ehrhart positivity, and dotted borders indicate classes containing
counterexamples to Ehrhart positivity.}
\label{fig:class-inclusions}
\end{figure}
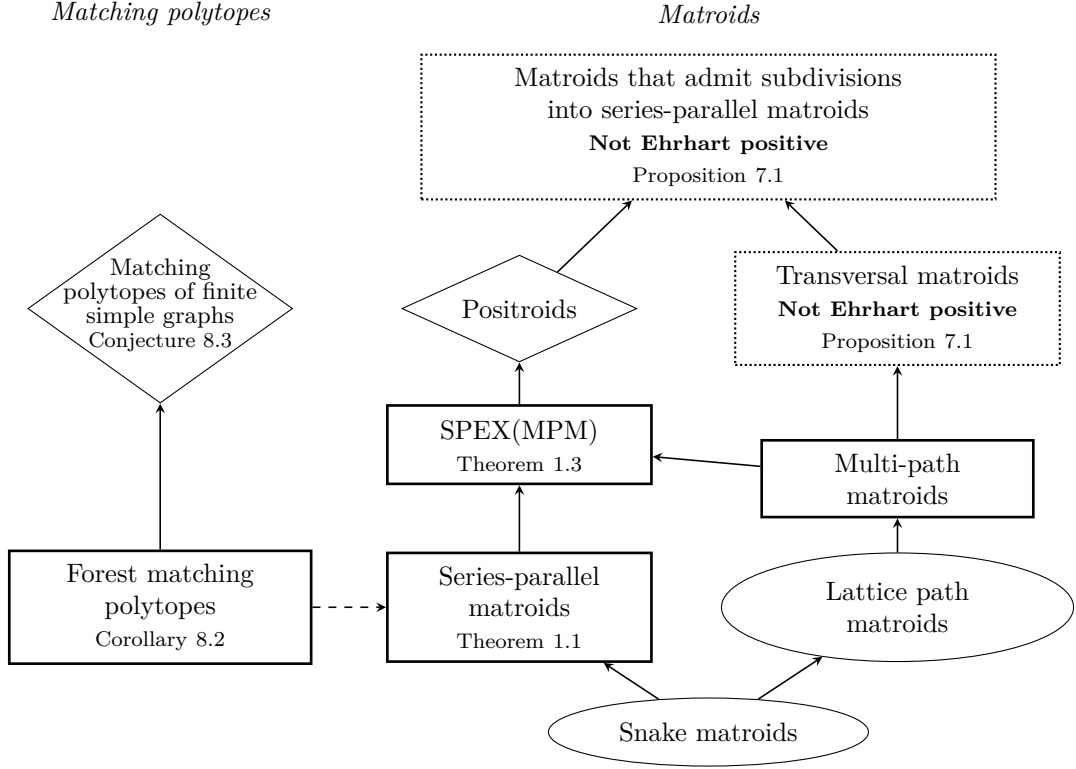

Section~\ref{subsec:algebraic-observation} outlines the algebraic
method behind our results. A class of polynomials with nonnegative
coefficients is shown to be closed under series and parallel
convolutions, providing a common argument for
Theorem~\ref{thm:main-introduction} and the extension criterion in
Lemma~\ref{lem:extension-criterion}.

The paper is organized as follows. Section~\ref{sec:preliminaries}
recalls matroid base polytopes and multi-path matroids.
Section~\ref{sec:series-parallel-trees} describes series-parallel
constructions by rooted trees and identifies their root polytopes
as matroid base polytopes. Section~\ref{sec:tree-counts} introduces the load polytopes and
derives the lattice-point recursions.
Section~\ref{sec:positive-expansions} proves the nonnegative expansion,
gives the marked-cycle formula for the Ehrhart coefficients,
and uses it to establish strict positivity. Section~\ref{sec:spex-mpm} proves the general criterion
in Lemma~\ref{lem:extension-criterion}. Weighted path counts and
a partition determined by prefix minima verify its hypothesis for
multi-path matroids, yielding the series-parallel extension result.
Section~\ref{sec:transversal-counterexample} gives a transversal
counterexample to the conjecture of Ferroni, Morales, and Panova.
Section~\ref{sec:matching-polytopes} proves Ehrhart positivity for
matching polytopes of forests and records further positive families
and a conjecture for arbitrary graphs.

\subsection{An algebraic observation on series-parallel extensions}
\label{subsec:algebraic-observation}

Ehrhart polynomials are multiplicative under direct sums. Since a
one-element matroid has Ehrhart polynomial \(1\), it suffices to prove
Theorem~\ref{thm:main-introduction} for connected series-parallel
matroids with at least two elements.
We encode a series-parallel construction by a tree rooted at the
\(\text{\upshape P}\)-labeled vertex of the initial two-leaf representation of
\(U_{1,2}\), where \(\text{\upshape P}\) denotes parallel composition;
see Section~\ref{subsec:sp-extensions}.
The tree determines lattice-point recursions in two parameters
\(a,b\), with \(a+b=t\).
By Theorem~\ref{thm:positive-expansion}, the resulting counting
polynomials are finite nonnegative rational linear combinations of
\begin{equation}
(a+b)^k\binom{a+c}{c}\binom{b+d}{d},
\qquad c,d,k\in\mathbb Z_{\ge0}.
\label{eq:intro-binomial-terms}
\end{equation}
The polynomials generated by \eqref{eq:intro-binomial-terms} with
nonnegative rational coefficients form a class closed under
parallel and series convolutions.
For parallel convolution, a beta integral representation and a
factorization of the integrand reduce closure to nonnegativity
of kernel coefficients as polynomials in \(t=a+b\).
A recurrence derived from a differential identity proves that the
kernel coefficients are nonnegative for every value of the
integration parameter in \([0,1]\).
Interchanging \(a,b\) gives closure under series convolution.

Evaluating the expansion of \(K_T(a,b)\) at \((t,0)\) or \((0,t)\)
proves Theorem~\ref{thm:main-introduction}.
Tracking which terms in \eqref{eq:intro-binomial-terms} occur with
positive weights further establishes strict positivity of the Ehrhart
coefficients in Section~\ref{sec:coefficients-examples}.
Under \((a,b)=(x_i,t-x_i)\), the factor
\(\binom{a+c}{c}\binom{b+d}{d}\) becomes
\(\binom{x_i+c}{c}\binom{t-x_i+d}{d}\), the weight for coordinate
\(i\) in \eqref{eq:atomic-weighted-count}.
Lemma~\ref{lem:single-extension-count} shows that summing over
the two coordinates introduced by one extension merges their
weights by a parallel or series convolution. Starting with weight
\(1\) on every final element and merging weights in reverse order
keeps every weight a nonnegative rational linear combination of
\eqref{eq:intro-binomial-terms}, by Lemma~\ref{lem:convolution-closure}.
Expanding the resulting weights on \(M\) expresses the Ehrhart
polynomial of the extension as a nonnegative rational linear
combination of \(t^kW_{M;\mathbf c,\mathbf d}(t)\).
If \(W_{M;\mathbf c,\mathbf d}\) has nonnegative coefficients for
every \(\mathbf c,\mathbf d\), then the extension is Ehrhart positive,
proving Lemma~\ref{lem:extension-criterion}.
Corollary~\ref{prop:marked-cycles} gives a weighted interpretation
of the Ehrhart coefficients in terms of permutations with marked cycles.
\FloatBarrier

\section{Preliminaries}\label{sec:preliminaries}

We recall basic operations on matroids, the corresponding properties
of their base polytopes, and the path representation of multi-path matroids.

\subsection{Matroid base polytopes}\label{subsec:direct-sums}

Let \(M\) be a matroid on a finite set \(E\).
An element of \(E\) is a \emph{loop} if it belongs to no basis,
and a \emph{coloop} if it belongs to every basis.
A \emph{circuit} is an inclusion-minimal subset of \(E\)
not contained in any basis, and a \emph{cocircuit} is an
inclusion-minimal subset meeting every basis.
For \(A\subseteq E\), the rank \(r_M(A)\) is the maximum of
\(|A\cap B|\) over all bases \(B\) of \(M\), and
\(\operatorname{rk}M=r_M(E)\).

For \(e\in E\), the \emph{deletion} \(M\mathbin{\backslash} e\) and the
\emph{contraction} \(M\mathbin{/}e\) are the matroids on \(E\mathbin{\backslash}\{e\}\)
with rank functions
\[
r_{M\mathbin{\backslash} e}(A)=r_M(A),\qquad
r_{M\mathbin{/}e}(A)=r_M(A\cup\{e\})-r_M(\{e\})
\]
for \(A\subseteq E\mathbin{\backslash}\{e\}\).
If \(e\) is not a coloop, the bases of \(M\mathbin{\backslash} e\) are the
bases of \(M\) avoiding \(e\). If \(e\) is not a loop, the bases
of \(M\mathbin{/}e\) are the sets \(B\mathbin{\backslash}\{e\}\), where \(B\) is a
basis of \(M\) containing \(e\).

The base polytope \(\mathcal P(M)\), defined in
\eqref{eq:matroid-base-polytope}, has the inequality description
\cite[Proposition~2.3]{FeichtnerSturmfelsBergman}
\begin{equation}
\mathcal P(M)=\Bigl\{x\in\mathbb R_{\ge0}^E:
\textstyle\sum_{e\in E}x_e=r_M(E),\quad
\sum_{e\in A}x_e\le r_M(A)\text{ for all }A\subseteq E\Bigr\}.
\label{eq:matroid-base-inequalities}
\end{equation}
Its vertices lie in \(\{0,1\}^E\), so it is a lattice polytope in
\(\mathbb Z^E\).

The bases of the dual matroid \(M^*\) are the complements in \(E\)
of the bases of \(M\). Since
\(\mathbf1_{E\mathbin{\backslash} B}=\mathbf1-\mathbf1_B\),
\[
\mathcal P(M^*)=\mathbf1-\mathcal P(M)
=\{\mathbf1-x:x\in\mathcal P(M)\},
\]
where \(\mathbf1\in\mathbb R^E\) is the all-ones vector;
see \cite[Remark~3.3]{FerroniNotEhrhartPositive}.

For matroids \(M_1,\ldots,M_r\) on pairwise disjoint ground sets,
the bases of their direct sum \(M=M_1\oplus\cdots\oplus M_r\)
are the unions \(B_1\cup\cdots\cup B_r\), where \(B_i\) is a
basis of \(M_i\) for each \(i\). Consequently,
\[
\mathcal P(M)=\mathcal P(M_1)\times\cdots\times\mathcal P(M_r).
\]
A matroid on a nonempty ground set is \emph{connected} if it is not
a direct sum of two matroids on nonempty ground sets. Every matroid
is the direct sum of its connected components. Let \(c(M)\) denote
the number of connected components of \(M\).
\begin{equation}
\dim\mathcal P(M)=|E|-c(M);
\label{eq:matroid-dimension}
\end{equation}
see \cite[Proposition~2.4]{FeichtnerSturmfelsBergman}.
In particular, a connected matroid on \(E\) has base-polytope
dimension \(|E|-1\).

\subsection{Multi-path matroids}\label{subsec:multipath-background}

We use the path representation of Bonin and Gim\'enez
\cite[Theorems~3.6 and~3.7]{BoninGimenezMultiPath}, written in terms
of prefix heights. Put \([n]=\{1,\ldots,n\}\).

\begin{definition}\label{def:multipath}
Let \(n\ge1\), \(0\le r\le n\), and \(w\ge0\) be integers.
Let \(L=(L_0,\ldots,L_n)\) and \(U=(U_0,\ldots,U_n)\) be
integer sequences satisfying
\begin{equation}
\begin{gathered}
L_0=0,\quad L_n=r,\qquad U_0=w,\quad U_n=w+r,\\
L_j-L_{j-1},\ U_j-U_{j-1}\in\{0,1\}\quad(1\le j\le n),
\qquad L_j\le U_j\quad(0\le j\le n).
\end{gathered}
\label{eq:multipath-data}
\end{equation}
The bases of \(N(L,U)\) are the \(r\)-element subsets \(B\subseteq[n]\)
for which some integer \(\theta\in\{0,\ldots,w\}\) satisfies
\begin{equation}
L_j\le |B\cap[j]|+\theta\le U_j\qquad(0\le j\le n),
\label{eq:multipath-bases}
\end{equation}
where \([0]=\varnothing\).
A \emph{multi-path matroid} is a matroid isomorphic to some \(N(L,U)\).
When \(w=0\), the matroid \(N(L,U)\) is a \emph{lattice path matroid},
which we denote by \(M(L,U)\). In this notation, \(L_j\) and
\(U_j\) are the heights of the lower and upper boundary paths
after \(j\) steps.
\end{definition}

The parameter \(\theta\) is the starting height of a path representing
\(B\). The \(j\)-th segment of the path is horizontal if
\(j\notin B\) and has slope one if \(j\in B\).
Different starting heights can represent the same basis.

For \(x\in\mathbb R^n\), put \(S_0(x)=0\) and
\(S_j(x)=x_1+\cdots+x_j\) for \(1\le j\le n\).
The path representation and the integrality of systems of
difference constraints give the following description of
the base polytope.

\begin{lemma}\label{lem:multipath-polytope}
For integer sequences \(L,U\) satisfying \eqref{eq:multipath-data}
and every real \(t\ge0\), the associated multi-path matroid
\(N(L,U)\) satisfies
\begin{equation}
t\mathcal P(N(L,U))=\left\{x\in[0,t]^n\,\middle|\,
\begin{aligned}
&S_n(x)=rt,\quad\text{and for some }\theta\in\mathbb R,\\
&L_jt\le S_j(x)+\theta\le U_jt\quad(0\le j\le n)
\end{aligned}\right\}.
\label{eq:multipath-polytope}
\end{equation}
\end{lemma}

\begin{proof}
Both sides are \(\{0\}\) at \(t=0\); scaling reduces \(t>0\)
to \(t=1\). Set \(z_j=S_j(x)+\theta\). The right-hand side is
the image, under \(x_j=z_j-z_{j-1}\), of the polytope defined by
\[
L_j\le z_j\le U_j\ (0\le j\le n),\qquad
0\le z_j-z_{j-1}\le1\ (1\le j\le n),\qquad z_n-z_0=r.
\]
After replacing the equality \(z_n-z_0=r\) by two inequalities,
each row of the constraint matrix contains at most one \(+1\)
and at most one \(-1\), with all other entries zero. Hence the
matrix is totally unimodular. Since all right-hand sides are
integers, every vertex of the polytope is integral
\cite[Theorems~4.1 and~4.4]{SchrijverPolyhedralCombinatorics}.
The map \(x_j=z_j-z_{j-1}\) sends every integral vertex of
this polytope to a basis indicator in
\eqref{eq:multipath-bases}, with \(\theta=z_0\).
Conversely, for a basis indicator \(x\), choose an integer
\(\theta\) satisfying \eqref{eq:multipath-bases}; then
\(z_j=S_j(x)+\theta\) satisfies the displayed constraints.
Thus the image under \(x_j=z_j-z_{j-1}\) is
\(\mathcal P(N(L,U))\).
\end{proof}

When \(w=0\), the condition at \(j=0\) forces \(\theta=0\),
and \eqref{eq:multipath-polytope} reduces to the lattice path
matroid polytope description of Knauer, Mart\'inez-Sandoval,
and Ram\'irez Alfons\'in
\cite[Theorems~3.3 and~3.4]{KnauerMartinezSandovalRamirezAlfonsin2018}.

\begin{example}\label{ex:multipath-whirl}
Take
\[
L=(0,1,1,2,2,3,3),\qquad U=L+\mathbf1.
\]
By Definition~\ref{def:multipath}, a three-element subset
\(B\subseteq[6]\) is a basis if and only if there exists
\(\theta\in\{0,1\}\) for which all six inequalities hold:
\[
\begin{alignedat}{2}
1&\le |B\cap[1]|+\theta\le2,\qquad&
1&\le |B\cap[2]|+\theta\le2,\\
2&\le |B\cap[3]|+\theta\le3,\qquad&
2&\le |B\cap[4]|+\theta\le3,\\
3&\le |B\cap[5]|+\theta\le4,\qquad&
3&\le |B\cap[6]|+\theta\le4.
\end{alignedat}
\]
The resulting rank-three matroid on \([6]\) is the three-whirl
\(\mathcal W^3\); its nonbasis triples are
\(\{1,2,6\}\), \(\{2,3,4\}\), and \(\{4,5,6\}\).
This matroid is not a lattice path matroid.
Figure~\ref{fig:multipath-prefix} shows the two boundary sequences.
The basis \(\{1,3,5\}\) is represented at both starting heights
\(\theta=0\) and \(\theta=1\).
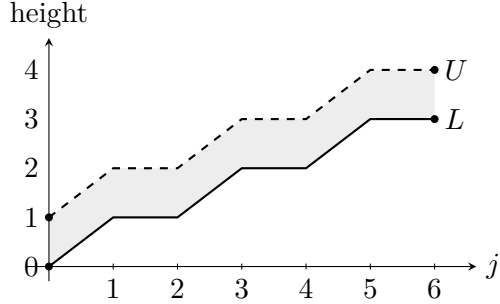
\begin{figure}[htbp]
\centering
\begin{tikzpicture}[x=0.85cm,y=0.65cm,>=stealth]
\fill[black!7] (0,0)--(1,1)--(2,1)--(3,2)--(4,2)--(5,3)--(6,3)
 --(6,4)--(5,4)--(4,3)--(3,3)--(2,2)--(1,2)--(0,1)--cycle;
\draw[->] (-0.35,0)--(6.65,0) node[right] {\(j\)};
\draw[->] (0,-0.3)--(0,4.65) node[above] {height};
\draw[thick] (0,0)--(1,1)--(2,1)--(3,2)--(4,2)--(5,3)--(6,3);
\draw[thick,dashed] (0,1)--(1,2)--(2,2)--(3,3)--(4,3)--(5,4)--(6,4);
\foreach \j in {1,...,6} {\draw (\j,0.06)--(\j,-0.06) node[below] {\(\j\)};}
\foreach \h in {0,...,4} {\node[left] at (0,\h) {\(\h\)};}
\foreach \point in {(0,0),(0,1),(6,3),(6,4)} {\fill \point circle (1.5pt);}
\node[right] at (6,3) {\(L\)};
\node[right] at (6,4) {\(U\)};
\end{tikzpicture}
\caption{The prefix-height representation in Example~\ref{ex:multipath-whirl}.}
\label{fig:multipath-prefix}
\end{figure}
\end{example}

\FloatBarrier

\section{\texorpdfstring{\text{\upshape S-P}}{S-P} trees}
\label{sec:series-parallel-trees}

We encode series and parallel extensions by finite rooted trees
whose leaves have pairwise distinct labels. Each internal vertex
has at least two unordered children and is labeled \(\text{\upshape P}\) or \(\text{\upshape S}\);
leaves carry no \(\text{\upshape P}\)- or \(\text{\upshape S}\)-label. We call such a tree an
\emph{\text{\upshape S-P} tree}. We describe basis indicators recursively on \text{\upshape S-P}
trees and identify their root polytopes as base polytopes of
connected series-parallel matroids.

\subsection{Series-parallel extensions and state vectors}\label{subsec:sp-extensions}\label{subsec:construction-trees}\label{ux6811ux7684ux53c2ux6570ux548cux4e24ux4e2aux6761ux4ef6ux72b6ux6001}

Two distinct elements are \emph{parallel} if they form a two-element
circuit and are \emph{in series} if they form a two-element cocircuit.
A parallel extension adjoins an element parallel to an existing one,
so deleting the new element recovers the old matroid. A series
extension is the dual operation: the new element is in series with
an existing one, and contracting it recovers the old matroid;
see \cite[Section~5, p.~1546]{SpeyerTropicalLinearSpaces}.

Let \(M\) be a connected matroid on \(E\) with at least two elements,
and fix \(e\in E\); then \(e\) is neither a loop nor a coloop.
Choose distinct new labels \(e_1,e_2\notin E\).
Relabel \(e\) as \(e_1\), adjoin \(e_2\) in parallel or in series
with it, and denote the resulting matroid by \(M'\). Its ground set is
\[
E'=(E\mathbin{\backslash}\{e\})\cup\{e_1,e_2\}.
\]
The bases of \(M'\) are given by the following table as \(B\) ranges
over the bases of \(M\), where an entry containing \(e_i\) is taken for
each \(i\in\{1,2\}\).
\begin{equation}
\begin{array}{c|cc}
\text{extension} & e\notin B & e\in B\\ \hline
\text{parallel} & B & (B\mathbin{\backslash}\{e\})\cup\{e_i\}\\[2pt]
\text{series} & B\cup\{e_i\} &
(B\mathbin{\backslash}\{e\})\cup\{e_1,e_2\}
\end{array}
\label{eq:extension-bases}
\end{equation}

We record a chosen sequence of series and parallel extensions by a
rooted binary tree. We begin with a \(\text{\upshape P}\)-labeled root and two leaves
representing \(U_{1,2}\). Each subsequent extension is recorded by replacing the
leaf corresponding to the extended element by a new internal vertex
with two leaf children. The new vertex is labeled \(\text{\upshape P}\) for a parallel
extension and \(\text{\upshape S}\) for a series extension. One leaf represents the
relabeled element, and the other represents the new element.
The final leaf labels form the ground set of the resulting matroid.
Figure~\ref{fig:construction-steps} illustrates two successive extensions.
The tree records a chosen construction and need not be unique.

\begin{figure}[ht]
\centering
\begin{tikzpicture}[
  internal/.style={circle,draw,thick,minimum size=7mm,inner sep=0pt},
  leaf/.style={draw,rounded corners=1pt,minimum width=9mm,
    minimum height=6mm,inner sep=1pt},
  every node/.style={font=\small}
]
\begin{scope}
  \node[internal] (r0) at (0,0) {\(\text{\upshape P}\)};
  \node[leaf] (e0) at (-0.8,-1.1) {\(e\)};
  \node[leaf] (f0) at (0.8,-1.1) {\(f\)};
  \draw (r0)--(e0) (r0)--(f0);
  \node at (0,-1.9) {\(U_{1,2}\)};
\end{scope}
\draw[->,thick] (1.15,0)--(3.65,0)
  node[midway,above=3pt,align=center,font=\scriptsize]
  {parallel extension\\at \(e\)};
\begin{scope}[xshift=48mm]
  \node[internal] (r1) at (0,0) {\(\text{\upshape P}\)};
  \node[internal] (p1) at (-0.8,-1.1) {\(\text{\upshape P}\)};
  \node[leaf] (f1) at (0.8,-1.1) {\(f\)};
  \node[leaf] (e11) at (-1.4,-2.2) {\(e_1\)};
  \node[leaf] (e12) at (-0.2,-2.2) {\(e_2\)};
  \draw (r1)--(p1) (r1)--(f1) (p1)--(e11) (p1)--(e12);
\end{scope}
\draw[->,thick] (5.95,0)--(8.65,0)
  node[midway,above=3pt,align=center,font=\scriptsize]
  {series extension\\at \(f\)};
\begin{scope}[xshift=98mm]
  \node[internal] (r2) at (0,0) {\(\text{\upshape P}\)};
  \node[internal] (p2) at (-1.1,-1.1) {\(\text{\upshape P}\)};
  \node[internal] (s2) at (1.1,-1.1) {\(\text{\upshape S}\)};
  \node[leaf] (e21) at (-1.7,-2.2) {\(e_1\)};
  \node[leaf] (e22) at (-0.55,-2.2) {\(e_2\)};
  \node[leaf] (f21) at (0.55,-2.2) {\(f_1\)};
  \node[leaf] (f22) at (1.7,-2.2) {\(f_2\)};
  \draw (r2)--(p2) (r2)--(s2)
    (p2)--(e21) (p2)--(e22) (s2)--(f21) (s2)--(f22);
\end{scope}
\end{tikzpicture}
\caption{A parallel extension at \(e\), followed by a series extension at \(f\).}
\label{fig:construction-steps}
\end{figure}
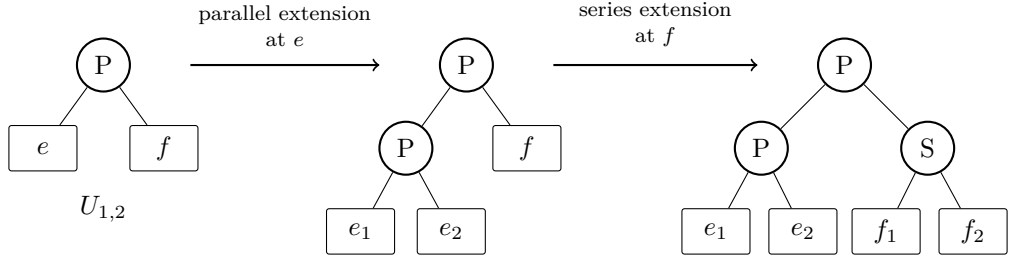

The construction above produces binary trees. We also allow vertices
with more than two children.

\begin{definition}\label{def:ps-tree}
For an \text{\upshape S-P} tree \(T\), the \emph{rooted subtree} \(T_v\) at a
vertex \(v\) consists of
\(v\), all its descendants, and all edges among these vertices.
Throughout this paper, \emph{subtree} always means a rooted subtree
of this form.

Suppose that the root has children \(v_1,\ldots,v_d\), and put
\(T_i=T_{v_i}\). We write \(T=\text{\upshape P}(T_1,\ldots,T_d)\) when the root is labeled
\(\text{\upshape P}\), and \(T=\text{\upshape S}(T_1,\ldots,T_d)\) when the root is labeled \(\text{\upshape S}\).
We write \(e\) for the tree consisting
of a single leaf labeled \(e\).
We write \(E(T)\) for the set of leaf labels and \(n_T=|E(T)|\)
for the number of leaves.
\end{definition}

\begin{definition}\label{def:tree-statistic}
For each \text{\upshape S-P} tree \(T\), define \(s(T)\in\mathbb Z_{\ge0}\) recursively by
\begin{equation}
\begin{aligned}
s(e)&=0,\\
s\bigl(\text{\upshape P}(T_1,\ldots,T_d)\bigr)&=\sum_{i=1}^{d}s(T_i),\\
s\bigl(\text{\upshape S}(T_1,\ldots,T_d)\bigr)&=\sum_{i=1}^{d}s(T_i)+d-1.
\end{aligned}
\label{eq:1.1}
\end{equation}
\end{definition}

\begin{example}\label{ex:tree-statistic}
Consider the tree \(T=\text{\upshape P}(T_1,T_2,T_3)\) in
Figure~\ref{fig:ps-tree}, where \(T_1=e_1\),
\(T_2=\text{\upshape S}(e_2,e_3)\), and \(T_3=\text{\upshape P}(e_4,e_5,e_6)\).
By \eqref{eq:1.1},
\[
s(T_1)=0,\qquad s(T_2)=1,\qquad s(T_3)=0
\quad\Longrightarrow\quad s(T)=1.
\]

\begin{figure}[ht]
\centering
\begin{tikzpicture}[
  level distance=13mm,
  level 1/.style={sibling distance=42mm},
  level 2/.style={sibling distance=14mm},
  internal/.style={circle,draw,thick,minimum size=7mm,inner sep=0pt},
  leaf/.style={draw,rounded corners=1pt,minimum width=10mm,
    minimum height=6mm,inner sep=1pt},
  edge from parent/.style={draw},
  every node/.style={font=\small}
]
\node[internal,label=above:{\scriptsize root}] (root) {\(\text{\upshape P}\)}
  child {node[leaf] (e1) {\(e_1\)}}
  child {node[internal] (s) {\(\text{\upshape S}\)}
    child {node[leaf] (e2) {\(e_2\)}}
    child {node[leaf] (e3) {\(e_3\)}}
  }
  child {node[internal] (p) {\(\text{\upshape P}\)}
    child {node[leaf] (e4) {\(e_4\)}}
    child {node[leaf] (e5) {\(e_5\)}}
    child {node[leaf] (e6) {\(e_6\)}}
  };
\node[
  draw=black,
  dashed,
  rounded corners=2pt,
  fit=(s)(e2)(e3),
  inner sep=3pt,
  label={[text=black,font=\small]below:{subtree \(T_2\)}}
] {};
\end{tikzpicture}
\caption{The tree \(T\) of Example~\ref{ex:tree-statistic}.
The dashed box indicates the rooted subtree \(T_2\).}
\label{fig:ps-tree}
\end{figure}
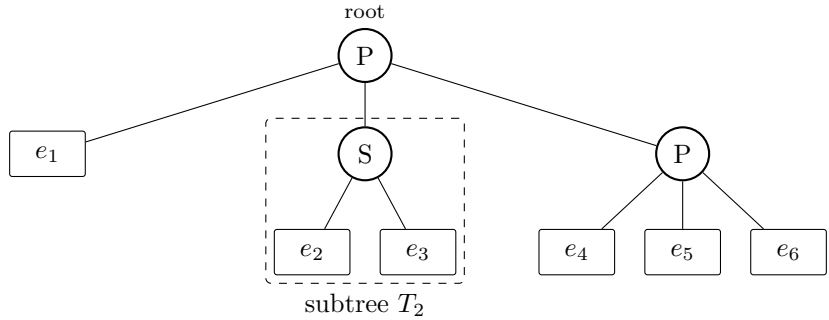
\end{example}

To describe bases recursively, we assign each vertex a \emph{state},
a value \(\varepsilon\in\{0,1\}\). If a subtree records extensions of
an element \(e\), its root state records whether \(e\) belongs to the
basis \(B\) before the extensions: \(\varepsilon=1\) when \(e\in B\),
and \(\varepsilon=0\) when \(e\notin B\).

\begin{definition}
\label{def:state-configurations}
For each \(\varepsilon\in\{0,1\}\), we recursively define a set
\(\mathcal B_T^\varepsilon\subseteq\{0,1\}^{E(T)}\), called the set of
state-\(\varepsilon\) vectors of \(T\). If \(T\) is the single leaf \(e\), set
\[
\mathcal B_e^\varepsilon
=\{x\in\{0,1\}^{\{e\}}:x_e=\varepsilon\}.
\]
Suppose that \(T\) has child subtrees \(T_1,\ldots,T_d\). Define
\begin{equation}
\mathcal B_T^\varepsilon=
\begin{cases}
\displaystyle
\bigcup_{\substack{(\varepsilon_1,\ldots,\varepsilon_d)\in\{0,1\}^d\\
\varepsilon_1+\cdots+\varepsilon_d=\varepsilon}}
\left\{x\in\{0,1\}^{E(T)}\,\middle|\,
\begin{array}{@{}l@{}}
x|_{E(T_i)}\in\mathcal B_{T_i}^{\varepsilon_i}\\
\text{for }i=1,2,\ldots,d
\end{array}\right\},
& T=\text{\upshape P}(T_1,\ldots,T_d),\\[3ex]
\displaystyle
\bigcup_{\substack{(\varepsilon_1,\ldots,\varepsilon_d)\in\{0,1\}^d\\
\varepsilon_1+\cdots+\varepsilon_d=d-1+\varepsilon}}
\left\{x\in\{0,1\}^{E(T)}\,\middle|\,
\begin{array}{@{}l@{}}
x|_{E(T_i)}\in\mathcal B_{T_i}^{\varepsilon_i}\\
\text{for }i=1,2,\ldots,d
\end{array}\right\},
& T=\text{\upshape S}(T_1,\ldots,T_d).
\end{cases}
\label{eq:state-vector-recursion}
\end{equation}
At a \(\text{\upshape P}\)-labeled root, state \(0\) requires all child states to be
\(0\), and state \(1\) requires exactly one child state to be \(1\).
At an \(\text{\upshape S}\)-labeled root, state \(0\) requires exactly one child state
to be \(0\), and state \(1\) requires all child states to be \(1\).

For each state tuple \((\varepsilon_1,\ldots,\varepsilon_d)\) satisfying
the corresponding sum condition in \eqref{eq:state-vector-recursion},
the restrictions \(x|_{E(T_i)}\in\mathcal B_{T_i}^{\varepsilon_i}\) may
be chosen independently for \(i=1,2,\ldots,d\).

For each \(\varepsilon\in\{0,1\}\), set
\[
\mathcal P_T^\varepsilon:=\operatorname{conv}(\mathcal B_T^\varepsilon)
\subseteq\mathbb R^{E(T)}.
\]
Define the \emph{root polytope} of \(T\) by
\[
\mathcal P(T):=
\begin{cases}
\mathcal P_T^1,&\text{if \(T\) is a leaf or its root is labeled }\text{\upshape P},\\
\mathcal P_T^0,&\text{if the root is labeled }\text{\upshape S}.
\end{cases}
\]
For the single-leaf tree \(e\), this gives
\(\mathcal P(e)=\{x:x_e=1\}\).
\end{definition}

For an \text{\upshape S-P} tree \(T\), we obtain a binary tree by combining the
child subtrees at each internal vertex two at a time, using the original
vertex's label for every binary combination. For example,
\[
\text{\upshape P}(T_1,T_2,T_3,T_4)\longmapsto \text{\upshape P}\bigl(\text{\upshape P}(T_1,T_2),\text{\upshape P}(T_3,T_4)\bigr).
\]
We call the resulting tree \(T'\) a \emph{binary bracketing} of \(T\).
Fix child states allowed by \eqref{eq:state-vector-recursion} at
an original vertex. At a \(\text{\upshape P}\)-vertex, assign each new intermediate
state the number of \(1\)'s among the original child states in its
group; at an \(\text{\upshape S}\)-vertex, assign it one minus the number of \(0\)'s
in its group. Each group contains at most one \(1\) in the
\(\text{\upshape P}\)-case or one \(0\) in the \(\text{\upshape S}\)-case, so every intermediate
state lies in \(\{0,1\}\). Conversely, eliminating the intermediate
states recovers \eqref{eq:state-vector-recursion} at the original vertex.
Thus
\[
\mathcal B_{T'}^\varepsilon=\mathcal B_T^\varepsilon
\qquad(\varepsilon\in\{0,1\}).
\]
An internal vertex with \(d\) children is replaced by \(d-1\) binary vertices
with the same label. By \eqref{eq:1.1}, the contribution to \(s\)
remains \(0\) for a \(\text{\upshape P}\)-vertex and \(d-1\) for an \(\text{\upshape S}\)-vertex.
Therefore \(s(T')=s(T)\). Since each binary \(\text{\upshape S}\)-vertex contributes
one, \(s(T')\) is the number of internal vertices of \(T'\) labeled
\(\text{\upshape S}\).

If \(T\) records a construction of \(M\) from \(U_{1,2}\), then
\(\mathcal P(T)=\mathcal P(M)\). In Section~\ref{subsec:matroid-polytopes}
we identify the matroid represented by \(\mathcal P(T)\) for arbitrary
\text{\upshape S-P} trees \(T\).

\subsection{Matroid base polytopes from \texorpdfstring{\text{\upshape S-P}}{S-P} trees}
\label{subsec:matroid-polytopes}\label{ux4e3aux4ec0ux4e48ux8986ux76d6ux5168ux90e8ux4e32ux5e76ux8054ux62dfux9635}\label{sec:9}

Given an \text{\upshape S-P} tree \(T\), we will realize \(\mathcal P_T^1\) and
\(\mathcal P_T^0\) as base polytopes of minors of a single matroid.
Choose an additional element \(e_0\notin E(T)\). We construct a
matroid \(\widehat M_T\) on \(E(T)\cup\{e_0\}\) whose deletion at
\(e_0\) has base polytope \(\mathcal P_T^1\) and whose contraction
at \(e_0\) has base polytope \(\mathcal P_T^0\).

Choose a binary bracketing \(T'\) of \(T\). We construct
\(\widehat M_T\) by performing the series and parallel extensions
prescribed by \(T'\), while leaving
\(e_0\) unchanged. Start with \(U_{1,2}\) on \(\{e_0,e\}\), and assign the temporary
element \(e\) to the root of \(T'\). Working from the root toward
the leaves, extend the element assigned to each internal vertex
according to its \(\text{\upshape P}\)- or \(\text{\upshape S}\)-label, and assign the two resulting elements
to its children. Give the final elements their leaf labels.

We identify the indicator vectors of the bases of \(\widehat M_T\) by
comparing the extension rules in \eqref{eq:extension-bases} with
the \(\text{\upshape P}\)- and \(\text{\upshape S}\)-state recursions in
\eqref{eq:state-vector-recursion}.
With the coordinate of the extended element as the parent state,
parallel and series extensions allow exactly the pairs of child
states in the binary \(\text{\upshape P}\)- and
\(\text{\upshape S}\)-recursions in \eqref{eq:state-vector-recursion}, respectively.

Since \eqref{eq:extension-bases} lists every basis after an extension,
every basis of \(\widehat M_T\) arises from one of the two initial bases.
Fix \(\varepsilon\in\{0,1\}\) and start with the basis of \(U_{1,2}\)
whose indicator vector is \((1-\varepsilon,\varepsilon)\), with the
\(e_0\)-coordinate first. At every extension, the two new coordinates
satisfy the child-state condition for the label of the extended vertex in
\eqref{eq:state-vector-recursion}. Hence the final leaf vector \(x\)
belongs to \(\mathcal B_{T'}^\varepsilon\). Since no extension changes
the \(e_0\)-coordinate, the full indicator vector is
\((1-\varepsilon,x)\).

Conversely, fix \(x\in\mathcal B_{T'}^\varepsilon\).
The recursive definition \eqref{eq:state-vector-recursion} provides
a choice of states at all vertices of \(T'\), with root state
\(\varepsilon\), leaf state \(x_e\) at each \(e\in E(T)\), and
child states satisfying the \(\text{\upshape P}\)- or \(\text{\upshape S}\)-sum condition at every
internal vertex. At each extension, \eqref{eq:extension-bases}
allows a new basis whose two new coordinates equal the prescribed
child states. Starting from \((1-\varepsilon,\varepsilon)\) and
performing the extensions from the root to the leaves therefore
produces a basis with indicator vector \((1-\varepsilon,x)\).

Taking \(\varepsilon=1,0\) and using
\(\mathcal B_{T'}^\varepsilon=\mathcal B_T^\varepsilon\), we conclude
that the indicator vectors of the bases of \(\widehat M_T\) are exactly
\[
\{(0,x):x\in\mathcal B_T^1\}
\;\cup\;
\{(1,x):x\in\mathcal B_T^0\}.
\]

By \eqref{eq:extension-bases}, every basis before an extension
produces at least one basis after the extension, and the
\(e_0\)-coordinate remains unchanged. Thus the initial bases
\(\{e\}\) and \(\{e_0\}\) produce final bases with indicator vectors
\((0,x)\) and \((1,y)\), where \(x\in\mathcal B_T^1\) and
\(y\in\mathcal B_T^0\), respectively. Hence
\(\mathcal B_T^1\ne\varnothing\) and
\(\mathcal B_T^0\ne\varnothing\). The final bases with indicator
vectors \((0,x)\) and \((1,y)\) avoid and contain \(e_0\), respectively,
so \(e_0\) is neither a coloop nor a loop.

Since \(e_0\) is not a coloop, the bases of
\(\widehat M_T\mathbin{\backslash} e_0\) are exactly the bases of
\(\widehat M_T\) avoiding \(e_0\).
The indicator vectors of the bases of \(\widehat M_T\) avoiding
\(e_0\) are precisely \((0,x)\) with \(x\in\mathcal B_T^1\).
Omitting the \(e_0\)-coordinate therefore identifies
\(\mathcal B_T^1\) as the set of indicator vectors of the bases
of \(\widehat M_T\mathbin{\backslash} e_0\).

Since \(e_0\) is not a loop, the bases of \(\widehat M_T\mathbin{/}e_0\)
are obtained by removing \(e_0\) from the bases of \(\widehat M_T\)
containing \(e_0\). On indicator vectors, removing \(e_0\) sends
\((1,x)\) to \(x\). Thus the indicator vectors of the bases of
\(\widehat M_T\mathbin{/}e_0\) form exactly \(\mathcal B_T^0\).
The definition of \(\mathcal P_T^\varepsilon\) and the base-polytope
formula \eqref{eq:matroid-base-polytope} give
\begin{equation}
\begin{aligned}
\mathcal P_T^1&=\operatorname{conv}(\mathcal B_T^1)
=\mathcal P(\widehat M_T\mathbin{\backslash} e_0),\\
\mathcal P_T^0&=\operatorname{conv}(\mathcal B_T^0)
=\mathcal P(\widehat M_T\mathbin{/}e_0).
\end{aligned}
\label{eq:state-minors}
\end{equation}
In \eqref{eq:state-minors}, the coordinates are indexed by the leaf set \(E(T)\).

Let \(x=(x_e)_{e\in E(T)}\in\mathcal B_T^\varepsilon\) be a
state-\(\varepsilon\) vector. We claim that
\begin{equation}
\sum_{e\in E(T)}x_e=s(T)+\varepsilon.
\label{eq:1.2}
\end{equation}
We prove \eqref{eq:1.2} by induction on \(T\).
For \(T=e\), the equality follows from \(s(e)=0\) and
\(x_e=\varepsilon\). Suppose that the root of \(T\) has child
subtrees \(T_1,\ldots,T_d\). Since \(x\in\mathcal B_T^\varepsilon\),
the recursion in \eqref{eq:state-vector-recursion} gives states
\(\varepsilon_1,\ldots,\varepsilon_d\in\{0,1\}\) such that
\[
x|_{E(T_i)}\in\mathcal B_{T_i}^{\varepsilon_i}
\qquad(1\le i\le d),
\]
where \(\sum_{i=1}^d\varepsilon_i=\varepsilon\) when the root of
\(T\) is labeled \(\text{\upshape P}\), and
\(\sum_{i=1}^d\varepsilon_i=d-1+\varepsilon\) when the root of
\(T\) is labeled \(\text{\upshape S}\).
By the induction hypothesis,
\(\sum_{e\in E(T)}x_e=\sum_{i=1}^d s(T_i)+\sum_{i=1}^d\varepsilon_i\).
Using the recursion for \(s(T)\) in \eqref{eq:1.1}, we obtain
\[
\sum_{e\in E(T)}x_e=
\begin{cases}
\displaystyle\sum_{i=1}^d s(T_i)+\varepsilon=s(T)+\varepsilon,
&\text{if the root of \(T\) is labeled }\text{\upshape P},\\[4pt]
\displaystyle\sum_{i=1}^d s(T_i)+d-1+\varepsilon=s(T)+\varepsilon,
&\text{if the root of \(T\) is labeled }\text{\upshape S}.
\end{cases}
\]

\begin{example}\label{ex:running-state-vectors}
For the tree \(T\) in Example~\ref{ex:tree-statistic},
\(\mathcal B_{T}^0\) consists of the indicator vectors of
\(\{e_2\}\) and \(\{e_3\}\), and \(\mathcal B_{T}^1\) consists of
the indicator vectors of the following nine subsets of \(E(T)\):
\[
\{e_2,e_3\}\qquad\text{and}\qquad
\{e_i,e_j\},\quad i\in\{2,3\},\quad j\in\{1,4,5,6\}.
\]
The vectors in \(\mathcal B_{T}^0\) and
\(\mathcal B_{T}^1\) have coordinate sums \(1\) and \(2\),
respectively, in agreement with \eqref{eq:1.2}.
\end{example}

For an \text{\upshape S-P} tree \(T\), construct \(M\) as follows. If \(T\) is a
single leaf, take \(M\) to be the single-coloop matroid on \(E(T)\),
so that \(\mathcal P(T)=\mathcal P(M)\). Otherwise, take a binary
bracketing \(T'\) of \(T\), so that
\(\mathcal P(T')=\mathcal P(T)\).
Set \(\varepsilon=1\) when the root of \(T'\) is labeled \(\text{\upshape P}\),
and \(\varepsilon=0\) when the root is labeled \(\text{\upshape S}\).
In both cases, \eqref{eq:state-vector-recursion} gives
\(\varepsilon_1+\varepsilon_2=1\). Thus the allowed pairs \((1,0)\)
and \((0,1)\) are precisely the indicator vectors of the two bases
of \(U_{1,2}\).
Assign the two elements of \(U_{1,2}\) to the two children of the
root of \(T'\), and perform the extensions encoded strictly below
the root, as in Section~\ref{subsec:construction-trees}.
Denote the resulting matroid on \(E(T)\) by \(M\).
The equivalence between basis extensions in \eqref{eq:extension-bases}
and child states in \eqref{eq:state-vector-recursion} identifies
\(\mathcal B_{T'}^\varepsilon\) as the set of indicator vectors of
the bases of \(M\).
Series and parallel extensions preserve connectedness, so \(M\)
is a connected series-parallel matroid and
\[
\mathcal P(M)=\operatorname{conv}(\mathcal B_{T'}^\varepsilon)
=\mathcal P(T')=\mathcal P(T).
\]
Conversely, Section~\ref{subsec:construction-trees} records a
construction of every connected series-parallel matroid with
at least two elements by an \text{\upshape S-P} tree.

\Needspace{5\baselineskip}
\begin{proposition}\label{prop:matroid-realization}
Let \(T\) be an \text{\upshape S-P} tree and \(M\) the connected series-parallel
matroid on \(E(T)\) constructed from \(T\) above.
Then \(\mathcal P(T)=\mathcal P(M)\).
\end{proposition}

\begin{remark}\label{rem:speyer-tree}
Speyer's representation
\cite[Proposition~5.1, p.~21]{SpeyerTropicalLinearSpacesPreprint}
uses an unrooted tree with labeled leaves and internal vertices of
degree three, each colored white or black. For an \text{\upshape S-P} tree \(T\)
with an internal root, take a binary bracketing \(T'\), remove its
root and the two incident edges, and join its two children by an edge.
Keep the leaf labels, color the remaining \(\text{\upshape P}\)-vertices white,
and color the \(\text{\upshape S}\)-vertices black. The resulting tree is Speyer's
representation of the matroid with base polytope \(\mathcal P(T)\);
see Figure~\ref{fig:speyer-correspondence}.
We retain a root to specify the child subtrees at each vertex,
so that the state-vector sets in \eqref{eq:state-vector-recursion}
and the lattice-point counts in Section~\ref{sec:tree-counts}
can be computed recursively from the leaves to the root.
\end{remark}

\begin{figure}[ht]
\centering
\begin{tikzpicture}[
  internal/.style={circle,draw,thick,minimum size=6.5mm,inner sep=0pt},
  leaf/.style={draw,rounded corners=1pt,minimum width=8mm,
    minimum height=5.5mm,inner sep=1pt},
  white vertex/.style={internal,fill=white},
  black vertex/.style={internal,fill=black},
  every node/.style={font=\small}
]
\node at (0,0.7) {Binary \text{\upshape S-P} tree};
\node[internal,dashed,fill=white]
  (r) at (0,0) {\(\text{\upshape P}\)};
\node[internal] (a) at (-1.7,-1.1) {\(\text{\upshape P}\)};
\node[internal] (b) at (1.4,-1.1) {\(\text{\upshape P}\)};
\node[internal] (s) at (-0.9,-2.2) {\(\text{\upshape S}\)};
\node[internal] (c) at (2.1,-2.2) {\(\text{\upshape P}\)};
\node[leaf] (e1) at (-2.4,-2.2) {\(e_1\)};
\node[leaf] (e2) at (-1.55,-3.3) {\(e_2\)};
\node[leaf] (e3) at (-0.25,-3.3) {\(e_3\)};
\node[leaf] (e4) at (0.8,-2.2) {\(e_4\)};
\node[leaf] (e5) at (1.5,-3.3) {\(e_5\)};
\node[leaf] (e6) at (2.7,-3.3) {\(e_6\)};
\draw[thick,dashed] (r)--(a) (r)--(b);
\draw (a)--(e1) (a)--(s) (s)--(e2) (s)--(e3)
  (b)--(e4) (b)--(c) (c)--(e5) (c)--(e6);
\draw[->,thick] (3.3,-1.25)--(4.9,-1.25)
  node[midway,above=4pt,align=center,font=\scriptsize]
  {remove root\\join children};
\begin{scope}[xshift=86mm]
  \node at (0,0.7) {Speyer's tree};
  \node[white vertex] (aa) at (-1.5,-1.1) {};
  \node[white vertex] (bb) at (0.7,-1.1) {};
  \node[black vertex] (ss) at (-1.5,-2.2) {};
  \node[white vertex] (cc) at (2,-2.2) {};
  \node[leaf] (f1) at (-2.7,-0.2) {\(e_1\)};
  \node[leaf] (f2) at (-2.15,-3.3) {\(e_2\)};
  \node[leaf] (f3) at (-0.85,-3.3) {\(e_3\)};
  \node[leaf] (f4) at (1.2,-0.2) {\(e_4\)};
  \node[leaf] (f5) at (1.35,-3.3) {\(e_5\)};
  \node[leaf] (f6) at (2.65,-3.3) {\(e_6\)};
  \draw[very thick] (aa)--(bb);
  \draw (aa)--(f1) (aa)--(ss) (ss)--(f2) (ss)--(f3)
    (bb)--(f4) (bb)--(cc) (cc)--(f5) (cc)--(f6);
\end{scope}
\end{tikzpicture}
\caption{From a binary bracketing of \(T\) in
Example~\ref{ex:tree-statistic} to Speyer's tree.
The root and its two incident edges, drawn with dashed lines on the
left, are replaced by the thick solid edge on the right.
White and black internal vertices represent
\(\text{\upshape P}\) and \(\text{\upshape S}\), respectively. The root on the left specifies the
child subtrees used in the lattice-point recursions.}
\label{fig:speyer-correspondence}
\end{figure}
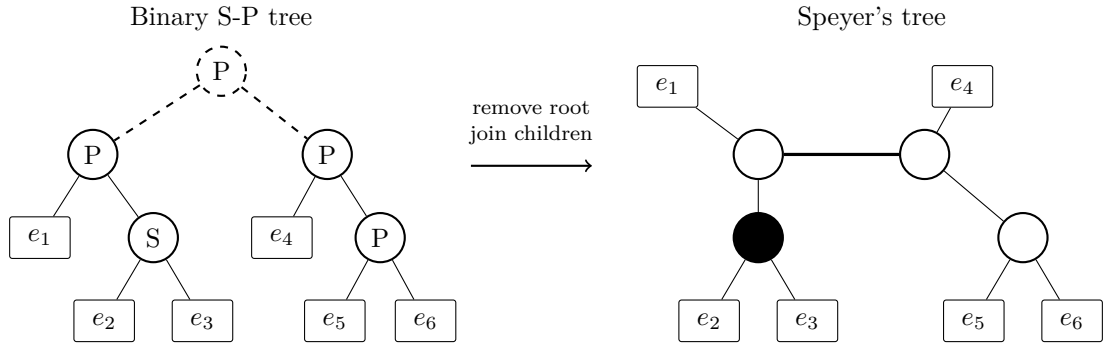

\begin{definition}\label{def:dual-tree}
For an \text{\upshape S-P} tree \(T\), let \(T^*\) be the tree obtained by
interchanging \(\text{\upshape P}\) and \(\text{\upshape S}\) at every internal vertex and keeping
all leaf labels unchanged.
\end{definition}

Replacing each leaf coordinate \(x_e\) by \(1-x_e\) interchanges
the \(\text{\upshape P}\)- and \(\text{\upshape S}\)-state recursions in
\eqref{eq:state-vector-recursion}, so
\[
\mathcal B_{T^*}^\varepsilon
=\{\mathbf1-x:x\in\mathcal B_T^{1-\varepsilon}\},
\qquad\varepsilon\in\{0,1\}.
\]
Taking convex hulls yields
\(\mathcal P_{T^*}^\varepsilon=\mathbf1-\mathcal P_T^{1-\varepsilon}\)
for \(\varepsilon\in\{0,1\}\).
If \(T\) is not a single leaf, Definition~\ref{def:state-configurations}
gives
\[
\mathcal P(T^*)=\mathbf1-\mathcal P(T),
\]
and \(\mathcal P(T^*)=\mathcal P(M^*)\) whenever
\(\mathcal P(T)=\mathcal P(M)\), by matroid duality.

The recursion \eqref{eq:1.1} also gives
\begin{equation}
s(T^*)=n_T-1-s(T).
\label{eq:tree-duality}
\end{equation}

\section{Load polytopes and lattice-point recursions}\label{sec:tree-counts}\label{ux6811ux6761ux4ef6ux72b6ux6001ux4e0eux5b9eux9645ux591aux9762ux4f53}\label{sec:1}

We describe the state polytopes by inequalities on coordinate sums over
subtrees. We then derive recursive formulas for lattice-point counts by
combining the counts on the child subtrees.

\subsection{Load polytopes}\label{ux8d1fux8f7dux7ea6ux675fux7684ux5fc5ux8981ux6027ux4e0eux5145ux5206ux6027}

For a subtree \(U\subseteq T\) and real \(t\ge0\), \eqref{eq:1.2}
shows that a vector in
\(t\mathcal B_U^\varepsilon\) has coordinate sum
\(s(U)t+\varepsilon t\). We therefore subtract \(s(U)t\) from
the sum of the coordinates on \(E(U)\). We call the adjusted sum the
\emph{load} on \(U\); on \(t\mathcal B_U^\varepsilon\), it equals
\(\varepsilon t\).

\begin{definition}
\label{def:load-polytope}
Let \(T\) be an \text{\upshape S-P} tree and let \(t\in\mathbb R_{\ge0}\). For each subtree
\(U\subseteq T\), define the affine load function
\(\lambda_U(\,\cdot\,;t):\mathbb R^{E(U)}\longrightarrow\mathbb R\) by
\begin{equation}
\lambda_U(x;t)=\sum_{e\in E(U)}x_e-s(U)t,
\qquad x\in\mathbb R^{E(U)}.
\label{eq:1.3}
\end{equation}
Define the polytope \(\mathcal Q_T(t)\) by
\begin{equation}
\mathcal Q_T(t):=\left\{x\in\mathbb R^{E(T)}:
0\le\lambda_U(x;t)\le t
\text{ for every subtree }U\subseteq T\right\}.
\label{eq:1.5}
\end{equation}
For a leaf subtree \(U=e\), the corresponding inequalities are
\(0\le x_e\le t\).
\end{definition}

Suppose that \(U\) has child subtrees \(U_1,\ldots,U_d\).
When the root of \(U\) is labeled \(\text{\upshape P}\), we have \(s(U)=\sum_i s(U_i)\);
when the root is labeled \(\text{\upshape S}\), we have
\(s(U)=\sum_i s(U_i)+d-1\). Substituting the formulas for \(s(U)\) into
\eqref{eq:1.3} gives
\begin{equation}
\begin{aligned}
\lambda_U(x;t)&=\sum_{i=1}^d\lambda_{U_i}(x;t)
&&\text{if }U=\text{\upshape P}(U_1,\ldots,U_d),\\
t-\lambda_U(x;t)&=\sum_{i=1}^d\bigl(t-\lambda_{U_i}(x;t)\bigr)
&&\text{if }U=\text{\upshape S}(U_1,\ldots,U_d).
\end{aligned}
\label{eq:1.4}
\end{equation}

\Needspace{9\baselineskip}
\begin{example}\label{ex:running-load-inequalities}
For the tree \(T\) in Example~\ref{ex:tree-statistic}, write
\(x_i=x_{e_i}\). The loads at the three
internal vertices are
\[
\lambda_{T_2}(x;t)=x_2+x_3-t,\qquad
\lambda_{T_3}(x;t)=x_4+x_5+x_6,\qquad
\lambda_{T}(x;t)=\sum_{i=1}^6x_i-t.
\]
Thus \(\mathcal Q_{T}(t)\) is given by \(0\le x_i\le t\) and
\[
t\le x_2+x_3\le2t,\qquad
0\le x_4+x_5+x_6\le t,\qquad
t\le\sum_{i=1}^6x_i\le2t.
\]
\end{example}

We will show that the inequalities
\[
0\le\lambda_U(x;t)\le t
\qquad\text{for every subtree }U\subseteq T
\]
characterize \(t\operatorname{conv}(\mathcal B_T^0\cup\mathcal B_T^1)\).

\begin{lemma}\label{lem:convexification}
For \(t\in\mathbb R_{\ge0}\) and
\(\varepsilon\in\{0,1\}\), we have
\begin{equation}
\mathcal Q_T(t)=t\operatorname{conv}(\mathcal B_T^0\cup\mathcal B_T^1),
\qquad
\mathcal Q_T(t)\cap\{x:\lambda_T(x;t)=\varepsilon t\}=t\mathcal P_T^\varepsilon.
\label{eq:1.6}
\end{equation}
\end{lemma}

\begin{proof}
If \(t=0\), the leaf inequalities \(0\le x_e\le t\) force
\(x=0\), and both equalities hold. For \(t>0\), the substitution
\(x=ty\) gives \(\lambda_U(x;t)=t\lambda_U(y;1)\) for every
subtree \(U\subseteq T\). Thus \(\mathcal Q_T(t)=t\mathcal Q_T(1)\),
and \(\lambda_T(x;t)=\varepsilon t\) is equivalent to
\(\lambda_T(y;1)=\varepsilon\). It therefore suffices to prove
the lemma for \(t=1\).

We first prove that \(\mathcal Q_T(1)\) is integral.
The inequalities defining \(\mathcal Q_T(1)\) are
\[
s(U)\le\sum_{e\in E(U)}x_e\le s(U)+1
\qquad\text{for every subtree }U\subseteq T.
\]
Let \(A=(A_{U,e})\) be the \(0/1\) matrix whose rows are indexed
by the subtrees \(U\subseteq T\) and whose columns are indexed by
the leaves \(e\in E(T)\), with \(A_{U,e}=1\) if and only if
\(e\in E(U)\). Order the columns by a depth-first traversal of the
leaves of \(T\). Each \(E(U)\) is then an interval, so every row
of \(A\) has consecutive ones. Hence \(A\) is totally unimodular
\cite[pp.~1671--1672]{SchrijverPolyhedralCombinatorics}; that is,
every square submatrix has determinant in \(\{0,\pm1\}\).
The lower bounds can be written as
\(-\sum_{e\in E(U)}x_e\le-s(U)\), so the full constraint matrix
is obtained by stacking \(A\) above \(-A\).
Duplicating rows and changing their signs preserve total unimodularity,
so the full constraint matrix is totally unimodular.
A polyhedron defined by a totally unimodular constraint matrix and an
integral right-hand side is integral
\cite[Theorem~4.1]{SchrijverPolyhedralCombinatorics}.
Here the right-hand sides \(-s(U)\) and \(s(U)+1\) are integers,
so \(\mathcal Q_T(1)\) is integral. The leaf inequalities
\(0\le x_e\le1\) then imply that every vertex is a \(0/1\) vector.

Let \(x\in\mathcal Q_T(1)\cap\{0,1\}^{E(T)}\).
For every subtree \(U\subseteq T\), the load
\(\lambda_U(x;1)\) lies in \(\{0,1\}\). Assign state
\(\lambda_U(x;1)\) to the root of each subtree \(U\subseteq T\).
At each leaf \(e\), the state is \(\lambda_e(x;1)=x_e\).
For a subtree \(U\) with child subtrees
\(U_1,\ldots,U_d\), \eqref{eq:1.4} gives
\[
\sum_{i=1}^d\lambda_{U_i}(x;1)=
\begin{cases}
\lambda_U(x;1),&\text{if }U=\text{\upshape P}(U_1,\ldots,U_d),\\
d-1+\lambda_U(x;1),&\text{if }U=\text{\upshape S}(U_1,\ldots,U_d).
\end{cases}
\]
Using \(\lambda_U(x;1)\) as the parent state and
\(\lambda_{U_i}(x;1)\) as the child states in
\eqref{eq:state-vector-recursion}, induction on \(U\) gives
\(x|_{E(U)}\in\mathcal B_U^{\lambda_U(x;1)}\) for every subtree
\(U\subseteq T\). Taking \(U=T\) yields
\(x\in\mathcal B_T^{\lambda_T(x;1)}\).

Conversely, if \(x\in\mathcal B_T^0\cup\mathcal B_T^1\),
Definition~\ref{def:state-configurations} gives
\(x|_{E(U)}\in\mathcal B_U^{\varepsilon_U}\) for some
\(\varepsilon_U\in\{0,1\}\) at every subtree \(U\subseteq T\).
Applying \eqref{eq:1.2} to \(x|_{E(U)}\) yields
\[
\lambda_U(x;1)=\sum_{e\in E(U)}x_e-s(U)=\varepsilon_U\in\{0,1\},
\]
so \(x\in\mathcal Q_T(1)\). We have proved
\[
\mathcal Q_T(1)\cap\{0,1\}^{E(T)}=\mathcal B_T^0\cup\mathcal B_T^1.
\]
Since every vertex of \(\mathcal Q_T(1)\) is a \(0/1\) vector,
\[
\mathcal Q_T(1)=\operatorname{conv}(\mathcal B_T^0\cup\mathcal B_T^1).
\]
For \(x\in\mathcal Q_T(1)\), write
\(x=\sum_{v\in\mathcal B_T^0\cup\mathcal B_T^1}\theta_v v\),
where \(\theta_v\ge0\) and \(\sum_v\theta_v=1\).
Since \(\lambda_T(\,\cdot\,;1)\) is affine,
\[
\lambda_T(x;1)
=\sum_{v\in\mathcal B_T^0\cup\mathcal B_T^1}\theta_v\lambda_T(v;1)
=\sum_{v\in\mathcal B_T^1}\theta_v.
\]
If \(\lambda_T(x;1)=0\), every coefficient \(\theta_v\) with
\(v\in\mathcal B_T^1\) is zero. If \(\lambda_T(x;1)=1\), the
coefficients indexed by \(\mathcal B_T^1\) already sum to \(1\),
so every \(\theta_v\) with \(v\in\mathcal B_T^0\) is zero.
Thus \(\lambda_T(x;1)=\varepsilon\) implies
\(x\in\operatorname{conv}(\mathcal B_T^\varepsilon)\).
Conversely, every convex combination of vectors in
\(\mathcal B_T^\varepsilon\) satisfies \(\lambda_T(x;1)=\varepsilon\).
Hence
\[
\mathcal Q_T(1)\cap\{x:\lambda_T(x;1)=\varepsilon\}
=\operatorname{conv}(\mathcal B_T^\varepsilon)
=\mathcal P_T^\varepsilon.\qedhere
\]
\end{proof}

\subsection{Recursions for lattice-point counts}\label{subsec:load-count}\label{ux7236ux8d1fux8f7dux8c31ux4e0eux7cbeux786eux9012ux5f52}\label{sec:2}

For lattice points \(x\in\mathcal Q_T(t)\), set
\(a=\lambda_T(x;t)\) and \(b=t-\lambda_T(x;t)\), so that
\(a,b\ge0\) and \(t=a+b\). If \(T\) has child subtrees
\(T_1,\ldots,T_d\), \eqref{eq:1.4} gives
\(a=\sum_i\lambda_{T_i}(x;t)\) at a \(\text{\upshape P}\)-root and
\(b=\sum_i(t-\lambda_{T_i}(x;t))\) at an \(\text{\upshape S}\)-root.
We therefore keep both parameters in the lattice-point count
\(K_T(a,b)\).

\begin{definition}\label{def:full-load-count}
For \(a,b\in\mathbb Z_{\ge0}\), put \(t=a+b\) and define
\begin{equation}
K_T(a,b)=\#\{x\in\mathcal Q_T(t)\cap\mathbb Z^{E(T)}:\lambda_T(x;t)=a\}.
\label{eq:2.1}
\end{equation}
\end{definition}

For \(t\in\mathbb Z_{\ge0}\), Lemma~\ref{lem:convexification} shows
that \(K_T(t,0)\) counts the lattice points of \(t\mathcal P_T^1\),
and \(K_T(0,t)\) counts the lattice points of \(t\mathcal P_T^0\).
The definition of \(\mathcal P(T)\) therefore gives
\begin{equation}
i(\mathcal P(T),t)=
\begin{cases}
K_T(t,0),&\text{if \(T\) is a single leaf or has a \(\text{\upshape P}\)-labeled root},\\
K_T(0,t),&\text{if \(T\) has an \(\text{\upshape S}\)-labeled root}.
\end{cases}
\label{eq:2.4}
\end{equation}

Next, we compute \(K_T\) recursively on nonnegative integer pairs.
Theorem~\ref{thm:positive-expansion} shows that \(K_T\) extends to a
polynomial in \(a\) and \(b\) with nonnegative rational coefficients.

\Needspace{11\baselineskip}
\begin{lemma}\label{lem:load-recursion}
Let \(a,b\in\mathbb Z_{\ge0}\). For a leaf \(e\), we have
\(K_e(a,b)=1\). Let \(U\) and \(V\) be \text{\upshape S-P} trees with disjoint sets of
leaf labels. If \(t=a+b\), then
\begin{align}
K_{\text{\upshape P}(U,V)}(a,b)&=\sum_{\zeta=0}^{a}
K_U(\zeta,t-\zeta)K_V(a-\zeta,b+\zeta),
\label{eq:2.2}\\
K_{\text{\upshape S}(U,V)}(a,b)&=\sum_{\zeta=0}^{b}
K_U(t-\zeta,\zeta)K_V(a+\zeta,b-\zeta).
\label{eq:2.3}
\end{align}
\end{lemma}

\begin{proof}
For a leaf, the unique coordinate is \(x_e=a\), so \(K_e(a,b)=1\).
A point \(x\) counted by \(K_{\text{\upshape P}(U,V)}(a,b)\) satisfies
\(\lambda_{\text{\upshape P}(U,V)}(x;t)=a\). The first identity in \eqref{eq:1.4}
therefore gives \(\lambda_U(x;t)+\lambda_V(x;t)=a\).
For each integer \(\zeta\) with \(0\le \zeta\le a\),
choose integer vectors \(y\in\mathcal Q_U(t)\) and
\(z\in\mathcal Q_V(t)\) with
\(\lambda_U(y;t)=\zeta\) and \(\lambda_V(z;t)=a-\zeta\). Form the vector
\(\widetilde{x}\) on \(E(\text{\upshape P}(U,V))\) by setting
\[
\widetilde{x}_e=\begin{cases}
y_e,&e\in E(U),\\
z_e,&e\in E(V).
\end{cases}
\]
For every subtree \(W\subseteq U\), we have
\(\lambda_W(\widetilde{x};t)=\lambda_W(y;t)\in[0,t]\), and for every
\(W\subseteq V\), we have \(\lambda_W(\widetilde{x};t)=\lambda_W(z;t)\in[0,t]\).
By the first identity in \eqref{eq:1.4},
\[
\lambda_{\text{\upshape P}(U,V)}(\widetilde{x};t)
=\lambda_U(y;t)+\lambda_V(z;t)
=\zeta+(a-\zeta)=a.
\]
Since \(t=a+b\) with \(a,b\ge0\), we have \(0\le a\le t\).
Together with the inequalities verified for the subtrees of \(U\) and
\(V\), this shows that \(\widetilde{x}\in\mathcal Q_{\text{\upshape P}(U,V)}(t)\).
Moreover, \(\lambda_{\text{\upshape P}(U,V)}(\widetilde{x};t)=a\), so \(\widetilde{x}\) is counted by
\(K_{\text{\upshape P}(U,V)}(a,b)\). Conversely, every point \(x\)
counted by \(K_{\text{\upshape P}(U,V)}(a,b)\) determines
\(y=(x_e)_{e\in E(U)}\), \(z=(x_e)_{e\in E(V)}\), and
\(\zeta=\lambda_U(x;t)\). There are
\(K_U(\zeta,t-\zeta)K_V(a-\zeta,b+\zeta)\) choices for \((y,z)\), which proves
\eqref{eq:2.2}.

For a point \(x\) counted by \(K_{\text{\upshape S}(U,V)}(a,b)\), set
\(\zeta=t-\lambda_U(x;t)\). Since \(\lambda_{\text{\upshape S}(U,V)}(x;t)=a\),
the second identity in \eqref{eq:1.4} gives
\(\zeta+t-\lambda_V(x;t)=t-a=b\). Thus \(0\le \zeta\le b\),
\(\lambda_U(x;t)=t-\zeta\), and \(\lambda_V(x;t)=a+\zeta\).
Conversely, for each integer \(\zeta\) with
\(0\le \zeta\le b\), choose integer vectors
\(y\in\mathcal Q_U(t)\) and \(z\in\mathcal Q_V(t)\) with
\(\lambda_U(y;t)=t-\zeta\) and \(\lambda_V(z;t)=a+\zeta\).
Set \(\widetilde{x}_e=y_e\) for \(e\in E(U)\) and \(\widetilde{x}_e=z_e\) for \(e\in E(V)\).
Then \(0\le\lambda_W(\widetilde{x};t)\le t\) for every subtree \(W\) contained
in \(U\) or \(V\), and
\(\lambda_{\text{\upshape S}(U,V)}(\widetilde{x};t)=(t-\zeta)+(a+\zeta)-t=a\).
The coordinates of \(\widetilde{x}\) recover \(y\) and \(z\) on their respective
leaf sets, while \(\zeta=t-\lambda_U(\widetilde{x};t)\). There are
\(K_U(t-\zeta,\zeta)K_V(a+\zeta,b-\zeta)\) choices for \((y,z)\), which proves
\eqref{eq:2.3}.
\end{proof}

Let \(a,b\in\mathbb Z_{\ge0}\) and \(t=a+b\).
For \(T=\text{\upshape P}(T_1,\ldots,T_d)\) and a point \(x\) counted by
\(K_T(a,b)\), put \(a_i=\lambda_{T_i}(x;t)\) for
\(i=1,\ldots,d\). Since \(\lambda_T(x;t)=a\), the first identity
in \eqref{eq:1.4} gives
\(a_1+\cdots+a_d=\lambda_T(x;t)=a\).
For each such tuple of nonnegative integers,
\(K_{T_i}(a_i,t-a_i)\) counts the integer vectors
\(y\in\mathcal Q_{T_i}(t)\) satisfying \(\lambda_{T_i}(y;t)=a_i\).
Combining the vectors on the disjoint leaf sets \(E(T_i)\) gives
\begin{equation}
K_T(a,b)=
\sum_{\substack{a_1,\ldots,a_d\in\mathbb Z_{\ge0}\\
                 a_1+\cdots+a_d=a}}
\prod_{i=1}^d K_{T_i}(a_i,t-a_i).
\label{eq:multiary-p-recursion}
\end{equation}
For \(T=\text{\upshape S}(T_1,\ldots,T_d)\), let \(x\) be a lattice point counted
by \(K_T(a,b)\), and put \(b_i=t-\lambda_{T_i}(x;t)\) for
\(i=1,\ldots,d\). The inequalities defining \(\mathcal Q_T(t)\)
imply that each \(b_i\) is a nonnegative integer. The second
identity in \eqref{eq:1.4} gives
\[
\sum_{i=1}^d b_i=t-\lambda_T(x;t)=t-a=b.
\]
Conversely, fix nonnegative integers \(b_1,\ldots,b_d\) with sum
\(b\). For each \(i=1,\ldots,d\), there are
\(K_{T_i}(t-b_i,b_i)\) lattice points
\(y^{(i)}\in\mathcal Q_{T_i}(t)\) satisfying
\(\lambda_{T_i}(y^{(i)};t)=t-b_i\).
The vectors \(y^{(i)}\) can be chosen independently and combined
on the disjoint leaf sets \(E(T_i)\). The resulting vector
\(\widetilde{x}\) satisfies the inequalities for every subtree
contained in a child subtree, and
\[
\lambda_T(\widetilde{x};t)=t-\sum_{i=1}^d b_i=a\in[0,t].
\]
Thus \(\widetilde{x}\) is counted by \(K_T(a,b)\). Its restrictions
to \(E(T_i)\) recover \(y^{(i)}\), and
\(b_i=t-\lambda_{T_i}(\widetilde{x};t)\). Summing the product of
the numbers of independent choices over all tuples gives
\begin{equation}
K_T(a,b)=
\sum_{\substack{b_1,\ldots,b_d\in\mathbb Z_{\ge0}\\
                 b_1+\cdots+b_d=b}}
\prod_{i=1}^d K_{T_i}(t-b_i,b_i).
\label{eq:multiary-s-recursion}
\end{equation}

In a binary bracketing of the \(d\) child subtrees, each new subtree
\(U\) has parameter
\(\lambda_U(x;t)=\sum_{i:T_i\subseteq U}a_i\) in the \(\text{\upshape P}\)-case, or
\(t-\lambda_U(x;t)=\sum_{i:T_i\subseteq U}b_i\) in the \(\text{\upshape S}\)-case.
The intermediate parameters are therefore uniquely determined.
Each group sum is nonnegative, at most its parent group sum, and at most
\(t\), so the new constraints \(0\le\lambda_U(x;t)\le t\) and
the binary summation ranges are satisfied. Thus each tuple occurs once,
and iterating \eqref{eq:2.2} and \eqref{eq:2.3} gives
\eqref{eq:multiary-p-recursion} and \eqref{eq:multiary-s-recursion},
respectively.

\begin{example}\label{ex:tree-load-polynomial}
For the tree \(T\) in Example~\ref{ex:tree-statistic},
let \(a,b\) be nonnegative integers and set \(t=a+b\). For integer
loads \(u,v\ge0\) on \(T_2,T_3\) with \(u+v\le a\), the
coordinate of \(e_1\) is \(a-u-v\). There are \(t-u+1\) integer
pairs \((x_{e_2},x_{e_3})\in[0,t]^2\) with
\(x_{e_2}+x_{e_3}=t+u\), and \(\binom{v+2}{2}\) nonnegative
integer triples \((x_{e_4},x_{e_5},x_{e_6})\) with sum \(v\).
Applying the \(\text{\upshape P}\)-recursion \eqref{eq:2.2} twice gives
\[
\begin{aligned}
K_{T}(a,b)
=\sum_{v=0}^a\binom{v+2}{2}\sum_{u=0}^{a-v}(t-u+1)
=\frac{4a+5b+5}{5}\binom{a+4}{4}.
\end{aligned}
\]

Since the root of \(T\) is labeled \(\text{\upshape P}\), 
\[
i(\mathcal P(T),t)=K_T(t,0)
=\frac{(4t+5)(t+1)(t+2)(t+3)(t+4)}{120}.
\]
At \(t=0\) this gives \(1\), and at \(t=1\) it gives \(9\),
the number of state-\(1\) vectors listed in
Example~\ref{ex:running-state-vectors}.
Each of the five linear factors has positive coefficients in
degrees \(0\) and \(1\), so their product has strictly positive
coefficients in degrees \(0\) through \(5=n_{T}-1\).
\end{example}

\section{Positive expansions and Ehrhart positivity}\label{sec:positive-expansions}\label{sec:ehrhart-positivity}\label{ux6839ux7aefux6240ux6709ux666eux901aux7cfbux6570ux7684ux4e25ux683cux6b63ux6027}\label{sec:8}\label{ux8981ux6784ux9020ux7684ux5b8cux6574ux5f31ux94feux6b63ux5c55ux5f00}\label{sec:3}

We seek an expansion of \(K_T(a,b)\) in which each summand has
nonnegative coefficients in \(a\) and \(b\).
Write \(\mathbb Q_{\ge0}[t]\) for the polynomials in \(t\) with
nonnegative rational coefficients.

\Needspace{11\baselineskip}
\begin{theorem}\label{thm:positive-expansion}
For every \text{\upshape S-P} tree \(T\), there exist nonnegative rational weights
\(w_T(c,d,k)\), indexed by \(c,d,k\in\mathbb Z_{\ge0}\) and
vanishing whenever \(c+d+k>n_T-1\), such that
\begin{equation}
K_T(a,b)=\sum_{c,d,k\ge0}w_T(c,d,k)(a+b)^k
\binom{a+c}{c}\binom{b+d}{d},
\label{eq:3.3}
\end{equation}
for all \(a,b\in\mathbb Z_{\ge0}\). Moreover,
\begin{equation}
w_T(p,q,0)>0,
\label{eq:8.1}
\end{equation}
where \(p\) and \(q\) count the internal vertices labeled
\(\text{\upshape P}\) and \(\text{\upshape S}\), respectively,
in any binary bracketing of \(T\).
\end{theorem}

The proof of Theorem~\ref{thm:positive-expansion} occupies
Sections~\ref{subsec:positive-kernel}--\ref{subsec:tree-induction}.
The proof of Proposition~\ref{thm:four-chain} expresses the \(\text{\upshape P}\)-convolution
as a beta integral. The factorization \eqref{eq:6.5} reduces coefficient
nonnegativity to the kernel studied in Section~\ref{subsec:positive-kernel}.
Section~\ref{subsec:tree-induction} applies the convolution formula
at each internal vertex and proves \(w_T(p,q,0)>0\) as in
\eqref{eq:8.1}. The terms with these weights will be used to prove
strict Ehrhart positivity in Section~\ref{sec:coefficients-examples}.

\subsection{Nonnegativity of the kernel coefficients}\label{subsec:positive-kernel}\label{ux4e8cux53d8ux91cfux6b63ux6838ux5f15ux7406}\label{sec:4}

For a polynomial or formal power series \(G(z)\), we write
\([z^h]G(z)\) for the coefficient of \(z^h\).
Work in \(\mathbb Q[t,u][[x,y]]\), with \(v=1-u\), and define
\begin{equation}
R(x,y)=\frac{u}{1-y}+\frac{v}{1-x}
=\frac{1-ux-vy}{(1-x)(1-y)}.
\label{eq:4.1}
\end{equation}
Since \(R-1\) has zero constant term, we interpret the formal power
\(R^{t+1}\) by the binomial expansion
\begin{equation}
R^{t+1}=\sum_{r\ge0}\binom{t+1}{r}(R-1)^r.
\label{eq:formal-binomial-power}
\end{equation}
In \eqref{eq:formal-binomial-power}, only the summands
\(\binom{t+1}{r}(R-1)^r\) with \(r\le m+n\) contribute to the
coefficient of \(x^my^n\) in \(R^{t+1}\). Write
\begin{equation}
F(t;x,y;u)=R(x,y)^{t+1}
=\sum_{m,n\ge0}F_{m,n}(t;u)x^my^n.
\label{eq:4.2}
\end{equation}

For nonnegative integers \(m,n\), \eqref{eq:formal-binomial-power}
gives \(F_{m,n}(t;u)\in\mathbb Q[t,u]\) and
\(\deg_t F_{m,n}\le m+n\). Also, \(F_{0,0}=1\).
Regarding \(F_{m,n}(t;u)\) as a polynomial in \(t\), write
\begin{equation}
F_{m,n}(t;u)=\sum_{j=0}^{m+n}g_{m,n,j}(u)t^j,
\qquad g_{m,n,j}(u)\in\mathbb Q[u].
\label{eq:kernel-t-coefficients}
\end{equation}
\begin{lemma}\label{lem:positive-kernel}
For nonnegative integers \(m,n\), the polynomials
\(g_{m,n,j}(u)\) satisfy
\[
g_{m,n,j}(u)\ge0
\qquad(0\le j\le m+n,\ 0\le u\le1).
\]
For \(0<u<1\), we have:
\begin{enumerate}[label=(\roman*)]
\item If \(m+n>0\), then \(g_{m,n,j}(u)>0\) for
\(1\le j\le m+n\).
\item \(g_{m,n,0}(u)>0\) if \(m=0\) or \(n=0\), whereas
\(g_{m,n,0}(u)=0\) if \(m,n>0\).
\end{enumerate}
\end{lemma}

\begin{proof}
Fix \(u\in[0,1]\). We prove the nonnegativity of
\(g_{m,n,j}(u)\), together with assertions (i) and (ii) when
\(0<u<1\), by induction on \(m+n\).

Setting \(y=0\) in \eqref{eq:4.1} gives
\(R(x,0)=(1-ux)/(1-x)\).

Substituting \(z=ux\) and \(z=x\) into the logarithmic series
\(\log(1-z)=-\sum_{r\ge1}z^r/r\) and subtracting gives
\[
\log R(x,0)=\log(1-ux)-\log(1-x)
=\sum_{r\ge1}\frac{1-u^r}{r}x^r.
\]
Hence
\begin{equation}
F(t;x,0;u)
=\exp\left((t+1)\sum_{r\ge1}\frac{1-u^r}{r}x^r\right).
\label{eq:kernel-boundary-exponential}
\end{equation}
Substituting \(z=(t+1)\sum_{r\ge1}(1-u^r)x^r/r\) into
\(\exp z=\sum_{k\ge0}z^k/k!\) and taking the coefficient of
\(x^m\) in \eqref{eq:kernel-boundary-exponential} gives
\begin{equation}
F_{m,0}(t;u)=\sum_{k=0}^{m}\frac{(t+1)^k}{k!}
[x^m]\left(\sum_{r\ge1}\frac{1-u^r}{r}x^r\right)^k.
\label{eq:kernel-boundary-coefficients}
\end{equation}
In \eqref{eq:kernel-boundary-coefficients}, the sum stops at \(k=m\):
the series \(\sum_{r\ge1}(1-u^r)x^r/r\) has zero constant term,
so its \(k\)-th power cannot contribute to \(x^m\) when \(k>m\).
Since \(0\le u\le1\), every coefficient \((1-u^r)/r\) is
nonnegative. Every power of \(\sum_{r\ge1}(1-u^r)x^r/r\) therefore has
nonnegative coefficients, so each coefficient extracted in
\eqref{eq:kernel-boundary-coefficients} is a nonnegative number.
Since \(k!>0\), this expresses \(F_{m,0}(t;u)\) as a nonnegative
linear combination of \((t+1)^k\). Each polynomial \((t+1)^k\) has
nonnegative coefficients in \(t\), so \(F_{m,0}(t;u)\) does as well.
If \(0<u<1\), the term \(k=m\) in
\eqref{eq:kernel-boundary-coefficients} is
\(v^m(t+1)^m/m!\), which has strictly positive coefficients
of \(t^0,\ldots,t^m\).
Interchanging \(x,y\) and \(u,v\) proves the assertions for
\(F_{0,n}\).

To treat \(m,n\ge1\), we take partial derivatives of \(R(x,y)\)
in \eqref{eq:4.1} with respect to \(x\) and \(y\) to obtain
\(R_{xy}=0\), \(R_x=v/(1-x)^2\), and \(R_y=u/(1-y)^2\).
Hence
\begin{equation}
\partial_x\partial_yF
=t(t+1)R^{t-1}R_xR_y
=\frac{t(t+1)uv}{(1-ux-vy)^2}F.
\label{eq:4.6}
\end{equation}
Substituting \(z=ux+vy\) into
\((1-z)^{-2}=\sum_{r\ge0}(r+1)z^r\), expanding \((ux+vy)^r\)
by the binomial theorem, and taking the coefficient of
\(x^{m-1}y^{n-1}\) in \eqref{eq:4.6} gives, for \(m,n\ge1\),
\begin{equation}
mnF_{m,n}(t;u)
=t(t+1)uv
\sum_{\substack{0\le h<m\\0\le k<n}}
(h+k+1)\binom{h+k}{h}u^hv^k
F_{m-1-h,n-1-k}(t;u).
\label{eq:kernel-coefficient-recurrence}
\end{equation}
Every polynomial \(F_{m-1-h,n-1-k}\) on the right has index sum
smaller than \(m+n\), so the induction hypothesis applies.
Since \(u,v\ge0\), the multiplier
\(t(t+1)(h+k+1)\binom{h+k}{h}u^{h+1}v^{k+1}\) also has
nonnegative coefficients in \(t\). Hence \(F_{m,n}\) has nonnegative
coefficients in \(t\).

Suppose now that \(0<u<1\). Every summand on the right of
\eqref{eq:kernel-coefficient-recurrence} has nonnegative
coefficients in \(t\). Retaining the term with \(h=m-1\) and
\(k=n-1\), and using \(F_{0,0}=1\), gives
\[
g_{m,n,j}(u)\ge
\frac{m+n-1}{mn}\binom{m+n-2}{m-1}u^mv^n>0
\qquad(j=1,2).
\]
For \(3\le j\le m+n\), the induction hypothesis gives
\(g_{m-1,n-1,j-2}(u)>0\), since
\(1\le j-2\le(m-1)+(n-1)\).
The \(h=k=0\) summand in \eqref{eq:kernel-coefficient-recurrence}
is \(uv\,t(t+1)F_{m-1,n-1}(t;u)\), so
\[
g_{m,n,j}(u)\ge\frac{uv}{mn}g_{m-1,n-1,j-2}(u)>0
\qquad(3\le j\le m+n).
\]

Finally, setting \(t=0\) in \eqref{eq:4.2} gives
\(F(0;x,y;u)=R(x,y)=\frac{u}{1-y}+\frac{v}{1-x}\).
The first fraction depends only on \(y\), and the second only on
\(x\). Hence, for \(m,n\ge1\),
\[
g_{m,n,0}(u)=[x^my^n]R(x,y)=0.
\]
Together with the strict positivity of \(g_{m,0,0}(u)\) and
\(g_{0,n,0}(u)\) established from
\eqref{eq:kernel-boundary-coefficients} for \({0<u<1}\), this proves (ii).
\end{proof}

\subsection{A beta integral for parallel convolution}\label{subsec:convolution}\label{ux4efbux610fux56dbux94feux6838ux7684ux6b63ux95edux5305}\label{sec:6}\label{subsec:beta-identity}\label{beta-ux6052ux7b49ux5f0fux7cbeux786eux7f16ux7801ux79bbux6563ux8d1fux8f7dux5206ux5272}\label{sec:5}

Let \(a,b,c,d,e,f\) be nonnegative integers and put \(t=a+b\).
The \(\text{\upshape P}\)-recursion \eqref{eq:2.2}, applied to
\(\binom{a+c}{c}\binom{b+d}{d}\) and
\(\binom{a+e}{e}\binom{b+f}{f}\), gives the sum
\begin{equation}
H_{c,d;e,f}(a,b)=\sum_{\zeta=0}^{a}
\binom{\zeta+c}{c}\binom{t-\zeta+d}{d}
\binom{a-\zeta+e}{e}\binom{b+\zeta+f}{f}.
\label{eq:6.1}
\end{equation}
We will express \(H_{c,d;e,f}(a,b)\) as an integral involving
the kernel \(F\). Define a probability measure \(\mu_{c,e}\) on \([0,1]\) by
\begin{equation}
\mathrm d\mu_{c,e}(u)=\frac{(c+e+1)!}{c!e!}u^c(1-u)^e\,\mathrm du.
\label{eq:5.1}
\end{equation}
For nonnegative integers \(A,B\), integration by parts gives
\[
\int_0^1u^A(1-u)^B\,\mathrm du
=\frac{A!B!}{(A+B+1)!}.
\]
For nonnegative integers \(p,q\), substitute \(A=c+p\) and
\(B=e+q\) into this formula and multiply by the normalization
factor \((c+e+1)!/(c!e!)\) in \eqref{eq:5.1}. This gives
\begin{equation}
\int_0^1u^p(1-u)^q\,\mathrm d\mu_{c,e}(u)
=\frac{(c+e+1)!(c+p)!(e+q)!}
{c!e!(c+e+p+q+1)!}>0.
\label{eq:5.2}
\end{equation}

\begin{proposition}\label{thm:four-chain}
For nonnegative integers \(c,d,e,f\), there are polynomials
\(A^{c,e}_{d,f;j}(t)\in\mathbb Q_{\ge0}[t]\), indexed by
\(0\le j\le d+f\), such that, for all nonnegative integers \(a,b\),
\begin{equation}
\begin{aligned}
H_{c,d;e,f}(a,b)
&=\sum_{j=0}^{d+f}A^{c,e}_{d,f;j}(a+b)
\binom{a+c+e+1}{c+e+1}\binom{b+j}{j},\\
\deg A^{c,e}_{d,f;j}&\le d+f-j.
\end{aligned}
\label{eq:6.2}
\end{equation}
\end{proposition}

\begin{proof}
With \(v=1-u\), we claim that
\begin{equation}
H_{c,d;e,f}(a,b)
=\binom{a+c+e+1}{c+e+1}[x^dy^f]
\int_0^1
\frac{(1-ux-vy)^a}{(1-x)^{t+1}(1-y)^{t+1}}\,
\mathrm d\mu_{c,e}(u).
\label{eq:beta-convolution-integral}
\end{equation}
To verify \eqref{eq:beta-convolution-integral}, write
\(1-ux-vy=u(1-x)+v(1-y)\). Expanding the \(a\)-th power
by the binomial theorem and using \(t=a+b\) gives
\[
\frac{(1-ux-vy)^a}{(1-x)^{t+1}(1-y)^{t+1}}
=\sum_{\zeta=0}^{a}\binom{a}{\zeta}
\frac{u^{\zeta}v^{a-\zeta}}
{(1-x)^{t-\zeta+1}(1-y)^{b+\zeta+1}}.
\]
For each integer \(\zeta\) with \(0\le\zeta\le a\), coefficient
extraction gives
\[
[x^dy^f](1-x)^{-(t-\zeta+1)}(1-y)^{-(b+\zeta+1)}
=\binom{t-\zeta+d}{d}\binom{b+\zeta+f}{f}.
\]
Applying \eqref{eq:5.2} with \(p=\zeta\) and \(q=a-\zeta\) gives
\[
\binom{a+c+e+1}{c+e+1}\binom{a}{\zeta}
\int_0^1u^{\zeta}v^{a-\zeta}\,\mathrm d\mu_{c,e}(u)
=\binom{\zeta+c}{c}\binom{a-\zeta+e}{e}.
\]
Multiplying the coefficient and moment identities and summing over
\(\zeta=0,\ldots,a\) recovers \eqref{eq:6.1}, proving
\eqref{eq:beta-convolution-integral}.

Using \(a=t-b\) and the definition of \(F\) in \eqref{eq:4.2},
we factor the integrand in \eqref{eq:beta-convolution-integral} as
\begin{equation}
\frac{(1-ux-vy)^a}{(1-x)^{t+1}(1-y)^{t+1}}
=F(t;x,y;u)(1-ux-vy)^{-b-1}.
\label{eq:6.5}
\end{equation}
The binomial theorem gives
\((1-ux-vy)^{-b-1}=\sum_{j\ge0}\binom{b+j}{j}(ux+vy)^j\).
Only the terms with \(0\le j\le d+f\) in the expansion of
\((1-ux-vy)^{-b-1}\) contribute to the coefficient of \(x^dy^f\)
after multiplication
by \(F\): every monomial in
\((ux+vy)^j\) has total degree \(j\), and \(F\) contains only
nonnegative powers of \(x,y\). For these values of \(j\), set
\begin{equation}
A^{c,e}_{d,f;j}(t)
=\int_0^1[x^dy^f](ux+vy)^jF(t;x,y;u)\,
\mathrm d\mu_{c,e}(u).
\label{eq:6.7}
\end{equation}
Substituting \(t=a+b\) and the binomial expansion of
\((1-ux-vy)^{-b-1}\) into \eqref{eq:beta-convolution-integral}
gives the expansion in
\eqref{eq:6.2}.

For \(0\le u\le1\), Lemma~\ref{lem:positive-kernel} and the
nonnegative coefficients of \((ux+vy)^j\) in \(x,y\) show that
\([x^dy^f](ux+vy)^jF(t;x,y;u)\) has nonnegative coefficients
in \(t\). Since the density of \(\mu_{c,e}\) is nonnegative on
\([0,1]\), integrating each coefficient in \(t\) gives a nonnegative
coefficient of \(A^{c,e}_{d,f;j}(t)\). The moment formula
\eqref{eq:5.2} gives rational coefficients. Thus
\(A^{c,e}_{d,f;j}(t)\in\mathbb Q_{\ge0}[t]\).

Only \(F_{m,n}(t;u)\) with \(m+n=d+f-j\) contribute to
\([x^dy^f](ux+vy)^jF\). Their degrees in \(t\) are at most
\(d+f-j\) by \eqref{eq:formal-binomial-power} and \eqref{eq:4.2}. Integration
with respect to \(u\) changes the coefficients but introduces
no new powers of \(t\). Hence
\(\deg A^{c,e}_{d,f;j}\le d+f-j\).
\end{proof}

\subsection{Convolution closure and the nonnegative expansion}\label{subsec:tree-induction}\label{ux5bf9ux4efbux610fux6811ux5b8cux6210ux7ed3ux6784ux5f52ux7eb3}\label{sec:7}

Motivated by the recursions for
\(K_{\text{\upshape P}(U,V)}\) and \(K_{\text{\upshape S}(U,V)}\)
in \eqref{eq:2.2} and \eqref{eq:2.3}, we define the parallel
and series convolution operators
\[
\text{\upshape P},\text{\upshape S}:
\mathbb Q[a,b]\times\mathbb Q[a,b]
\longrightarrow\mathbb Q[a,b].
\]
For \(G,H\in\mathbb Q[a,b]\) and \(a,b\in\mathbb Z_{\ge0}\),
put \(t=a+b\) and set
\[
\begin{aligned}
\text{\upshape P}(G,H)(a,b)
&=\sum_{\zeta=0}^{a}G(\zeta,t-\zeta)\,H(a-\zeta,b+\zeta),\\
\text{\upshape S}(G,H)(a,b)
&=\sum_{\zeta=0}^{b}G(t-\zeta,\zeta)\,H(a+\zeta,b-\zeta).
\end{aligned}
\]
By the polynomiality of power sums, both finite sums extend uniquely
to polynomials in \(a,b\).

\begin{definition}\label{def:weak-chain-polynomials}
Let \(\mathcal C_{a,b}\subseteq\mathbb Q[a,b]\) be the class
\begin{equation}
\mathcal C_{a,b}=
\left\{
\sum_{i=1}^{N}\alpha_i(a+b)^{k_i}
\binom{a+c_i}{c_i}\binom{b+d_i}{d_i}
\;\middle|\;
\substack{N\in\mathbb Z_{\ge1},\ \alpha_i\in\mathbb Q_{\ge0}\\
c_i,d_i,k_i\in\mathbb Z_{\ge0}\ (1\le i\le N)}
\right\}.
\label{eq:3.2}
\end{equation}
\end{definition}

By Definition~\ref{def:weak-chain-polynomials}, \(\mathcal C_{a,b}\)
is closed under multiplication by \((a+b)^k\) for \(k\in\mathbb Z_{\ge0}\)
and finite nonnegative rational linear combinations.
For convenience, write
\[
\Theta_{a,b}(r,s,k)=(a+b)^k\binom{a+r}{r}\binom{b+s}{s},
\qquad r,s,k\in\mathbb Z_{\ge0}.
\]

\begin{lemma}\label{lem:convolution-closure}
The class \(\mathcal C_{a,b}\) is closed under both convolutions:
\begin{equation}
\text{\upshape P}(\mathcal C_{a,b},\mathcal C_{a,b})\subseteq\mathcal C_{a,b},
\qquad
\text{\upshape S}(\mathcal C_{a,b},\mathcal C_{a,b})\subseteq\mathcal C_{a,b}.
\label{eq:7.1}
\end{equation}
\end{lemma}

\begin{proof}
We first compute \(\text{\upshape P}(G,H)\) for
\(G=\Theta_{a,b}(c,d,k_1)\) and \(H=\Theta_{a,b}(e,f,k_2)\),
where \(c,d,e,f,k_1,k_2\) are nonnegative integers.
For \(t=a+b\), taking the factor \(t^{k_1+k_2}\) outside the
convolution sum gives \(\text{\upshape P}(G,H)(a,b)=t^{k_1+k_2}H_{c,d;e,f}(a,b)\).
Expanding each coefficient polynomial \(A^{c,e}_{d,f;j}(t)\) in
\eqref{eq:6.2} in powers of \(t\) gives the output terms
\begin{equation}
\Theta_{a,b}(c+e+1,j,k_1+k_2+\ell)
\quad\text{with weight}\quad [t^\ell]A^{c,e}_{d,f;j}(t),
\label{eq:parallel-output-rule}
\end{equation}
where \(0\le j\le d+f\) and \(0\le\ell\le d+f-j\).
Proposition~\ref{thm:four-chain} ensures that the weights in
\eqref{eq:parallel-output-rule} are nonnegative rational numbers.

Define \(\iota:\mathbb Q[a,b]\to\mathbb Q[a,b]\) by
\((\iota F)(a,b)=F(b,a)\).
Since \(\iota\Theta_{a,b}(c,d,k)=\Theta_{a,b}(d,c,k)\),
the map \(\iota\) preserves \(\mathcal C_{a,b}\).
For arbitrary \(G,H\in\mathbb Q[a,b]\), substituting the exchanged
arguments into the defining convolution sums gives
\begin{equation}
\text{\upshape S}(G,H)=\iota\,\text{\upshape P}(\iota G,\iota H).
\label{eq:7.2}
\end{equation}
For \(G=\Theta_{a,b}(c,d,k_1)\) and \(H=\Theta_{a,b}(e,f,k_2)\),
we have \(\iota G=\Theta_{a,b}(d,c,k_1)\) and
\(\iota H=\Theta_{a,b}(f,e,k_2)\).
Applying \eqref{eq:parallel-output-rule} to \(\iota G,\iota H\),
and then applying \(\iota\) to each output term as in
\eqref{eq:7.2}, gives
\begin{equation}
\Theta_{a,b}(j,d+f+1,k_1+k_2+\ell)
\quad\text{with weight}\quad [t^\ell]A^{d,f}_{c,e;j}(t),
\label{eq:series-output-rule}
\end{equation}
where \(0\le j\le c+e\) and \(0\le\ell\le c+e-j\).
The weights in \eqref{eq:series-output-rule} are likewise
nonnegative rational numbers.

For general \(G,H\in\mathcal C_{a,b}\), choose finite expansions
\[
G=\sum_{i=1}^{N}\alpha_i\Theta_{a,b}(c_i,d_i,k_i),
\qquad
H=\sum_{h=1}^{M}\beta_h\Theta_{a,b}(e_h,f_h,\ell_h),
\]
where \(N,M\) are positive integers,
\(\alpha_i,\beta_h\in\mathbb Q_{\ge0}\), and
\(c_i,d_i,k_i,e_h,f_h,\ell_h\in\mathbb Z_{\ge0}\)
for \(1\le i\le N\) and \(1\le h\le M\).
Substitute the expansions into the convolution sums and distribute
the products. For each pair \(i,h\), convolve
\(\Theta_{a,b}(c_i,d_i,k_i)\) and \(\Theta_{a,b}(e_h,f_h,\ell_h)\)
using \eqref{eq:parallel-output-rule} for \(\text{\upshape P}\)
or \eqref{eq:series-output-rule} for \(\text{\upshape S}\),
and multiply each resulting weight by \(\alpha_i\beta_h\).
There are \(NM\) pairs of input terms, and each pair produces
finitely many terms with nonnegative rational weights after
multiplication by \(\alpha_i\beta_h\).
Thus \(\text{\upshape P}(G,H)\) and \(\text{\upshape S}(G,H)\)
are finite nonnegative rational linear combinations of the
polynomials \(\Theta_{a,b}(r,s,k)\), so both belong to \(\mathcal C_{a,b}\).
\end{proof}

\begin{proof}[Proof of Theorem~\ref{thm:positive-expansion}]
Fix a binary bracketing of \(T\). We compute the weights
\(w_T(c,d,k)\) by induction on the number of leaves.
For a leaf \(e\), the identity
\(K_e(a,b)=1\) gives the expansion with \(w_e(0,0,0)=1\) and
all other weights zero. Suppose that the counts \(K_U,K_V\) for
the two child subtrees \(U,V\) have expansions of the form
\eqref{eq:3.3}, with nonnegative rational weights \(w_U\) and \(w_V\)
of finite support.

The counting recursions \eqref{eq:2.2} and \eqref{eq:2.3} give
\(K_{\text{\upshape P}(U,V)}=\text{\upshape P}(K_U,K_V)\) and
\(K_{\text{\upshape S}(U,V)}=\text{\upshape S}(K_U,K_V)\), respectively.
Substitute the expansions of \(K_U,K_V\) into
\(\text{\upshape P}(K_U,K_V)\) and \(\text{\upshape S}(K_U,K_V)\),
and apply \eqref{eq:parallel-output-rule} and \eqref{eq:series-output-rule}
to each pair of summands.
Collecting the output terms with the same triple \((r,s,k)\) gives
\[
K_T(a,b)=\sum_{r,s,k\ge0}w_T(r,s,k)\Theta_{a,b}(r,s,k).
\]
For \(T=\text{\upshape P}(U,V)\), the resulting weights are
\[
w_T(r,s,k)=
\sum_{\substack{c,d,e,f,k_1,k_2,\ell\ge0\\
c+e+1=r,\ k_1+k_2+\ell=k\\
s+\ell\le d+f}}
w_U(c,d,k_1)w_V(e,f,k_2)
[t^\ell]A^{c,e}_{d,f;s}(t).
\]
For \(T=\text{\upshape S}(U,V)\), the resulting weights are
\[
w_T(r,s,k)=
\sum_{\substack{c,d,e,f,k_1,k_2,\ell\ge0\\
d+f+1=s,\ k_1+k_2+\ell=k\\
r+\ell\le c+e}}
w_U(c,d,k_1)w_V(e,f,k_2)
[t^\ell]A^{d,f}_{c,e;r}(t).
\]
All summation indices are nonnegative integers.
The finite expansions of \(K_U,K_V\) and the bounds on the output
indices in \eqref{eq:parallel-output-rule} and \eqref{eq:series-output-rule}
ensure that only finitely many \(w_T(r,s,k)\) are nonzero.
Every summand in the formulas for \(w_T(r,s,k)\) is a nonnegative
rational number by Proposition~\ref{thm:four-chain}.

By \eqref{eq:multiary-p-recursion} and
\eqref{eq:multiary-s-recursion}, the expansion obtained from the
binary bracketing also represents \(K_T\) for the original tree \(T\).

For the degree bound, we induct on the number of leaves.
For a leaf \(e\), the only nonzero weight is \(w_e(0,0,0)=1\),
and \(0+0+0=n_e-1\).
For the inductive step, suppose \(w_U(c,d,k_1)w_V(e,f,k_2)>0\).
By induction, \(c+d+k_1\le n_U-1\) and \(e+f+k_2\le n_V-1\).
If \(T=\text{\upshape P}(U,V)\), a nonzero summand in the formula for \(w_T(r,s,k)\)
has \(r=c+e+1\), \(k=k_1+k_2+\ell\), and \(s+\ell\le d+f\), so
\begin{equation}
\begin{aligned}
r+s+k&=(c+e+1)+s+k_1+k_2+\ell\\
&\le c+d+e+f+k_1+k_2+1\\
&\le n_U+n_V-1=n_T-1.
\end{aligned}
\label{eq:7.3}
\end{equation}
For \(T=\text{\upshape S}(U,V)\), a nonzero summand has
\(s=d+f+1\), \(k=k_1+k_2+\ell\), and \(r+\ell\le c+e\),
which likewise give \(r+s+k\le n_T-1\).
Induction therefore proves that
\(w_T(c,d,k)=0\) whenever \(c+d+k>n_T-1\).

To prove \eqref{eq:8.1}, we induct on the number of leaves in
the chosen binary bracketing.
For a leaf, the claim follows from \(w_e(0,0,0)=1\).
For the inductive step, let \(U,V\) be the two child subtrees.
Let \(c\) and \(d\) be the numbers of internal vertices of \(U\)
labeled \(\text{\upshape P}\) and \(\text{\upshape S}\), respectively, and define \(e,f\) in
the same way for \(V\). Suppose inductively
that \(w_U(c,d,0)>0\) and \(w_V(e,f,0)>0\).
At \(T=\text{\upshape P}(U,V)\), take \(j=d+f\) in \eqref{eq:6.7}
so that \(A^{c,e}_{d,f;d+f}(t)\) is a positive constant and contributes
to \(w_T(c+e+1,d+f,0)\) without increasing the exponent of \(a+b\).
The polynomial \((ux+(1-u)y)^{d+f}\) has total degree \(d+f\),
so extracting \(x^dy^f\) from \((ux+(1-u)y)^{d+f}F(t;x,y;u)\)
uses only the constant term \(F_{0,0}=1\). Hence
\[
A^{c,e}_{d,f;d+f}(t)
=\binom{d+f}{d}\int_0^1u^d(1-u)^f\,\mathrm d\mu_{c,e}(u)>0,
\]
where positivity follows from \eqref{eq:5.2}.
Taking \(j=d+f\) and \(\ell=0\) in the parallel output rule
\eqref{eq:parallel-output-rule} gives
\[
w_T(c+e+1,d+f,0)
\ge w_U(c,d,0)w_V(e,f,0)
[t^0]A^{c,e}_{d,f;d+f}(t)>0.
\]
The tree \(T=\text{\upshape P}(U,V)\) has \(c+e+1\) internal vertices labeled
\(\text{\upshape P}\) and \(d+f\) labeled \(\text{\upshape S}\).

\Needspace{7\baselineskip}
For \(T=\text{\upshape S}(U,V)\), the same integral calculation with
\((c,d,e,f)\) replaced by \((d,c,f,e)\) gives
\([t^0]A^{d,f}_{c,e;c+e}(t)>0\).
Taking \(j=c+e\) and \(\ell=0\) in \eqref{eq:series-output-rule} gives
\[
w_T(c+e,d+f+1,0)
\ge w_U(c,d,0)w_V(e,f,0)
[t^0]A^{d,f}_{c,e;c+e}(t)>0.
\]
The tree \(T=\text{\upshape S}(U,V)\) has \(c+e\) internal vertices labeled
\(\text{\upshape P}\) and \(d+f+1\) labeled \(\text{\upshape S}\).
Thus \(w_T(p,q,0)>0\), proving \eqref{eq:8.1}.

Thus \eqref{eq:3.3} represents \(K_T\) by a polynomial in
\(\mathcal C_{a,b}\subseteq\mathbb Q_{\ge0}[a,b]\).
\end{proof}

\Needspace{10\baselineskip}
\begin{proof}[Proof of Theorem~\ref{thm:main-introduction}]
Suppose first that \(M\) is connected and has at least two elements.
Choose an \text{\upshape S-P} tree \(T\) recording a construction of \(M\) as in
Section~\ref{subsec:construction-trees}.
Proposition~\ref{prop:matroid-realization} gives
\(\mathcal P(M)=\mathcal P(T)\).
Setting \((a,b)=(t,0)\) and \((0,t)\) in \eqref{eq:3.3} gives
\begin{equation}
\begin{aligned}
K_T(t,0)&=\sum_{c,d,k\ge0}w_T(c,d,k)t^k\binom{t+c}{c},\\
K_T(0,t)&=\sum_{c,d,k\ge0}w_T(c,d,k)t^k\binom{t+d}{d}.
\end{aligned}
\label{eq:10.2}
\end{equation}
Every summand has nonnegative coefficients in \(t\).
By \eqref{eq:2.4}, \(i(M,t)\) equals \(K_T(t,0)\) at a
\(\text{\upshape P}\)-root and \(K_T(0,t)\) at a
\(\text{\upshape S}\)-root, so \(M\) is Ehrhart positive.
Multiplicativity under direct sums completes the proof, since
one-element components have Ehrhart polynomial \(1\).
\end{proof}

\subsection{Coefficient formulas and strict positivity}\label{sec:coefficients-examples}\label{ux666eux901aux7cfbux6570ux7684ux65e0ux7b26ux53f7ux6709ux7406ux52a0ux6743ux7ec4ux5408ux6a21ux578b}\label{sec:10}

We first express the Ehrhart coefficients as a sum of nonnegative
weights, then exhibit a positive summand in each degree.
For a nonnegative integer \(c\), let \(S_c\) be the symmetric group on \(c\)
elements, and let \(\operatorname{Cyc}(\sigma)\) be the set of
cycles of \(\sigma\in S_c\). The group \(S_0\) consists of the
empty permutation, with no cycles. The standard enumeration by
cycles \cite[Proposition~1.3.7]{StanleyEC1} gives
\begin{equation}
\binom{t+c}{c}=\frac1{c!}
\sum_{\sigma\in S_c}(t+1)^{|\operatorname{Cyc}(\sigma)|}.
\label{eq:10.3}
\end{equation}

\begin{corollary}\label{prop:marked-cycles}
Let \(T\) be an \text{\upshape S-P} tree with weights \(w_T(c,d,k)\) from
Theorem~\ref{thm:positive-expansion}, and let \(h\ge0\) be an integer.
If \(T\) is a leaf or its root is labeled \(\text{\upshape P}\), then
\begin{equation}
[t^h]i(\mathcal P(T),t)
=\sum_{\substack{c,d,k\ge0,\;\sigma\in S_c\\
A\subseteq\operatorname{Cyc}(\sigma),\;k+|A|=h}}
\frac{w_T(c,d,k)}{c!}.
\label{eq:10.4}
\end{equation}
For an \(\text{\upshape S}\)-labeled root, replace \(S_c\) and \(c!\)
in \eqref{eq:10.4} by \(S_d\) and \(d!\), respectively.
\end{corollary}

\begin{proof}
In \((t+1)^{|\operatorname{Cyc}(\sigma)|}\), choosing \(t\)
from the factors indexed by \(A\) and \(1\) from the others
gives \(t^{|A|}\). Substitute \eqref{eq:10.3} into the expansion
of \(K_T(t,0)\) in \eqref{eq:10.2} and extract the coefficient of
\(t^h\), using \eqref{eq:2.4}.
For an \(\text{\upshape S}\)-labeled root, use the expansion of
\(K_T(0,t)\) in \eqref{eq:10.2}, with \(\binom{t+d}{d}\)
in place of \(\binom{t+c}{c}\).
For a leaf, the only nonzero weight is \(w_T(0,0,0)=1\),
and the empty permutation gives \eqref{eq:10.4}.
\end{proof}

\begin{corollary}\label{thm:strict-root}
For every \text{\upshape S-P} tree \(T\),
\begin{equation}
[t^h]i(\mathcal P(T),t)>0\qquad(0\le h\le n_T-1).
\label{eq:8.2}
\end{equation}
\end{corollary}

\begin{proof}
For a single leaf, \(i(\mathcal P(T),t)=1\).
Suppose that the root is labeled \(\text{\upshape P}\), and fix a
binary bracketing of \(T\) with child subtrees \(U,V\) at the root.
Use the weights constructed for this bracketing in the proof of
Theorem~\ref{thm:positive-expansion}.
Put \(c=n_U-1-s(U)\), \(d=s(U)\), \(e=n_V-1-s(V)\), and
\(f=s(V)\). Put \(p=c+e+1\) and \(q=d+f\), so
\(p+q=n_T-1\). Equation~\eqref{eq:8.1} gives
\(w_U(c,d,0)>0\), \(w_V(e,f,0)>0\), and \(w_T(p,q,0)>0\).

For \(1\le k\le q=d+f\), Lemma~\ref{lem:positive-kernel} gives
\(g_{d,f,k}(u)>0\) for \(0<u<1\).
Taking \(j=0\) in \eqref{eq:6.7} and using the positive density
of \(\mu_{c,e}\) on \((0,1)\) yields
\[
[t^k]A^{c,e}_{d,f;0}(t)
=\int_0^1g_{d,f,k}(u)\,\mathrm d\mu_{c,e}(u)>0.
\]
Applying \eqref{eq:parallel-output-rule} with \(j=0\) and \(\ell=k\) to
\(\Theta_{a,b}(c,d,0)\) and \(\Theta_{a,b}(e,f,0)\) therefore gives
\begin{equation}
w_T(p,0,k)\ge w_U(c,d,0)w_V(e,f,0)
[t^k]A^{c,e}_{d,f;0}(t)>0
\qquad(1\le k\le q).
\label{eq:strict-positive-weights}
\end{equation}

In \eqref{eq:10.4}, take \(\sigma\) to be the identity permutation
in \(S_p\), which has \(p\) singleton cycles.
For \(0\le h\le p\), use the weight \(w_T(p,q,0)\) and choose
any \(h\)-element subset \(A\) of \(\operatorname{Cyc}(\sigma)\).
The resulting summand is \(w_T(p,q,0)/p!>0\).
For \(p<h\le p+q\), use \(w_T(p,0,h-p)\) and
\(A=\operatorname{Cyc}(\sigma)\), obtaining a contribution
\(w_T(p,0,h-p)/p!>0\).
All other summands in \eqref{eq:10.4} are nonnegative, proving
\eqref{eq:8.2} for a \(\text{\upshape P}\)-labeled root.

If the root is labeled \(\text{\upshape S}\), then \(T^*\) has a
\(\text{\upshape P}\)-labeled root and the same number of leaves.
By \(\mathcal P(T^*)=\mathbf1-\mathcal P(T)\), proved after
Definition~\ref{def:dual-tree}, the map \(x\mapsto t\mathbf1-x\)
is a lattice-point bijection from \(t\mathcal P(T)\) to
\(t\mathcal P(T^*)\).
Thus \(i(\mathcal P(T),t)=i(\mathcal P(T^*),t)\), and the
\(\text{\upshape P}\)-root case proves \eqref{eq:8.2}.
\end{proof}

\section{Ehrhart positivity under series-parallel extensions}\label{sec:spex-mpm}

We prove the extension criterion in Lemma~\ref{lem:extension-criterion}
using weighted lattice-point counts. We then use weighted counts for
half-open lattice path matroid polytopes to establish Ehrhart positivity
for multi-path matroids and their series-parallel extensions
(Theorem~\ref{thm:spex-mpm}).

\subsection{An extension criterion for Ehrhart positivity}\label{subsec:tree-substitution}

For a matroid \(M\) on a finite set \(E\) and a family
\(\mathbf f=(f_e)_{e\in E}\) of polynomials in \(\mathbb Q[a,b]\), define
\begin{equation}
W_{M;\mathbf f}(t)=
\sum_{x\in t\mathcal P(M)\cap\mathbb Z^E}
\prod_{e\in E}f_e(x_e,t-x_e).
\label{eq:weighted-mpm-count}
\end{equation}
For \(E=[n]\), taking
\(f_i(a,b)=\binom{a+c_i}{c_i}\binom{b+d_i}{d_i}\)
for \(i=1,\ldots,n\) gives
\(W_{M;\mathbf f}(t)=W_{M;\mathbf c,\mathbf d}(t)\)
from \eqref{eq:atomic-weighted-count}.

To express the dependence of the weight on \(t\) through an
additional coordinate, set
\(\widehat{\mathcal P}=\mathcal P(M)\times\{1\}\) and
\(w(x,s)=\prod_{e\in E}f_e(x_e,s-x_e)\).
Since \(t\widehat{\mathcal P}=\{(x,t):x\in t\mathcal P(M)\}\),
\eqref{eq:weighted-mpm-count} becomes
\[
W_{M;\mathbf f}(t)=
\sum_{(x,s)\in t\widehat{\mathcal P}\cap(\mathbb Z^E\times\mathbb Z)}
w(x,s).
\]
The weighted Ehrhart theorem
\cite[Corollary~5.2(a)]{DeLoeraValenciaVillarrealWangWeighted}
applied to \(\widehat{\mathcal P}\) and \(w\) gives the polynomiality
of \(W_{M;\mathbf f}\).

Suppose that \(M'\) is obtained from a connected matroid \(M\)
on \(E\), with \(|E|\ge2\), by a parallel or series extension
at \(e\in E\). Write
\(E'=(E\mathbin{\backslash}\{e\})\cup\{e_1,e_2\}\) for the new
ground set, as in \eqref{eq:extension-bases}.
For a family \(\mathbf g=(g_j)_{j\in E'}\) of polynomials in \(\mathbb Q[a,b]\),
set \(f_j=g_j\) for \(j\in E\mathbin{\backslash}\{e\}\) and
\begin{equation}
f_e=
\begin{cases}
\text{\upshape P}(g_{e_1},g_{e_2}),&\text{for a parallel extension},\\
\text{\upshape S}(g_{e_1},g_{e_2}),&\text{for a series extension}.
\end{cases}
\label{eq:single-extension-weights}
\end{equation}

\begin{lemma}\label{lem:single-extension-count}
For every \(t\in\mathbb Z_{\ge0}\),
\begin{equation}
W_{M';\mathbf g}(t)=W_{M;\mathbf f}(t).
\label{eq:single-extension-count}
\end{equation}
\end{lemma}

\begin{proof}
We first show that the points of \(t\mathcal P(M')\) are
parametrized uniquely by \(x\in t\mathcal P(M)\) and a real
number \(\zeta\) as follows. Put \(a=x_e\) and \(b=t-a\),
keep \(y_j=x_j\) for \(j\in E\mathbin{\backslash}\{e\}\), and set
\begin{equation}
(y_{e_1},y_{e_2})=
\begin{cases}
(\zeta,a-\zeta),\quad 0\le\zeta\le a,&\text{for a parallel extension},\\
(t-\zeta,a+\zeta),\quad 0\le\zeta\le b,&\text{for a series extension}.
\end{cases}
\label{eq:single-extension-coordinates}
\end{equation}
At \(t=1\), the parallel basis description in
\eqref{eq:extension-bases} shows that the map fixing the
\(E\mathbin{\backslash}\{e\}\)-coordinates and setting
\(x_e=y_{e_1}+y_{e_2}\) sends \(\mathcal P(M')\) into
\(\mathcal P(M)\). Since \(y_{e_1},y_{e_2}\ge0\), taking
\(\zeta=y_{e_1}\) gives \(0\le\zeta\le x_e=a\).
Conversely, write \(x=\sum_{B\in\mathcal B(M)}\alpha_B\mathbf1_B\),
where \(\alpha_B\ge0\) and \(\sum_B\alpha_B=1\).
If \(a=x_e>0\), replace each term indexed by \(B\ni e\) with
the indicators of \((B\mathbin{\backslash}\{e\})\cup\{e_1\}\)
and \((B\mathbin{\backslash}\{e\})\cup\{e_2\}\), with coefficients
\(\alpha_B\zeta/a\) and \(\alpha_B(a-\zeta)/a\), respectively.
Retain each term indexed by \(B\not\ni e\).
Since \(\sum_{B\ni e}\alpha_B=a\), the resulting convex combination
has new coordinates \((\zeta,a-\zeta)\) and all other coordinates
equal to those of \(x\), proving \(y\in\mathcal P(M')\).
If \(a=0\), only bases avoiding \(e\) have positive coefficients,
and \(y_{e_1}=y_{e_2}=0\).
Scaling gives the parallel case of
\eqref{eq:single-extension-coordinates} for \(t>0\); at \(t=0\),
both dilates consist of the origin.

For a series extension, \(M'^*\) is a parallel extension of
\(M^*\). Apply the parallel case to \(t\mathbf1-y\) and
\(t\mathbf1-x\). The coordinate relation becomes
\(t-x_e=(t-y_{e_1})+(t-y_{e_2})\), or
\(x_e=y_{e_1}+y_{e_2}-t\).
Taking \(\zeta=t-y_{e_1}\) gives the series pair and the interval
\(0\le\zeta\le t-x_e=b\) in \eqref{eq:single-extension-coordinates}.

To recover \(x\) and \(\zeta\) from \(y\), take
\(x_j=y_j\) for \(j\in E\mathbin{\backslash}\{e\}\), and set
\((x_e,\zeta)=(y_{e_1}+y_{e_2},y_{e_1})\) in the parallel case
or \((x_e,\zeta)=(y_{e_1}+y_{e_2}-t,t-y_{e_1})\) in the series case.
Thus \(x\) and \(\zeta\) are uniquely determined.
For integral \(t\), the vector \(y\) is integral
if and only if \(x\) and \(\zeta\) are integral.
Summing \(g_{e_1}(y_{e_1},t-y_{e_1})g_{e_2}(y_{e_2},t-y_{e_2})\)
over the integer values of \(\zeta\) in
\eqref{eq:single-extension-coordinates} gives \(f_e(a,b)\), by
\eqref{eq:single-extension-weights} and the definitions of
\(\text{\upshape P},\text{\upshape S}\). Therefore
\[
W_{M';\mathbf g}(t)=
\sum_{x\in t\mathcal P(M)\cap\mathbb Z^E}
f_e(x_e,t-x_e)\prod_{j\in E\mathbin{\backslash}\{e\}}g_j(x_j,t-x_j)
=W_{M;\mathbf f}(t).\qedhere
\]
\end{proof}

\begin{proof}[Proof of Lemma~\ref{lem:extension-criterion}]
Let \(M'\) be obtained from \(M\) by a finite sequence of series
and parallel extensions. Assign weight \(1\) to every element
of \(M'\), so \(i(M',t)=W_{M';\mathbf1}(t)\).
Series and parallel extensions preserve connectedness.
Apply \eqref{eq:single-extension-count} from the last extension
to the first, merging each pair of weights by
\eqref{eq:single-extension-weights}.
Since \(1\in\mathcal C_{a,b}\) and Lemma~\ref{lem:convolution-closure}
gives closure under \(\text{\upshape P}\) and \(\text{\upshape S}\),
we obtain
\[
i(M',t)=W_{M;\mathbf f}(t),\qquad
f_i\in\mathcal C_{a,b}\quad(1\le i\le n).
\]
\Needspace{8\baselineskip}
For \(i=1,\ldots,n\), choose a finite expansion
\(f_i=\sum_{c,d,k\ge0}w_i(c,d,k)\Theta_{a,b}(c,d,k)\),
where \(w_i(c,d,k)\in\mathbb Q_{\ge0}\).
Substituting \(a=x_i\), \(b=t-x_i\) and expanding the product in
\eqref{eq:weighted-mpm-count} gives
\begin{equation}
W_{M;\mathbf f}(t)=
\sum_{\mathbf c,\mathbf d,\mathbf k\in\mathbb Z_{\ge0}^n}
\left(\prod_{i=1}^n w_i(c_i,d_i,k_i)\right)
t^{\sum_{i=1}^n k_i}W_{M;\mathbf c,\mathbf d}(t).
\label{eq:weighted-count-expansion}
\end{equation}
Every \(W_{M;\mathbf c,\mathbf d}\) has nonnegative coefficients
by hypothesis, and every \(w_i(c_i,d_i,k_i)\) is nonnegative.
Hence \(i(M',t)\in\mathbb Q_{\ge0}[t]\).
\end{proof}

\subsection{Weighted counts for lattice path matroids}\label{subsec:weighted-path-counts}

We use the description in Lemma~\ref{lem:multipath-polytope}
to count lattice points in dilates of lattice path matroid polytopes with
selected inequalities \(S_j(x)\ge L_jt\) made strict.
To impose strict lower inequalities, the recursion in
Lemma~\ref{lem:weighted-open-path} uses the
\(\text{\upshape S}\)-convolution with range \(0\le\zeta<b\).
The class \(\mathcal C_{a,b}\) is not closed under the
\(\text{\upshape S}\)-convolution restricted to \(0\le\zeta<b\):
taking both inputs to be \(1\) gives \(\sum_{\zeta=0}^{b-1}1=b\).
At \(b=0\), each generator
\((a+b)^k\binom{a+c}{c}\binom{b+d}{d}\) equals
\(a^k\binom{a+c}{c}>0\) for \(a>0\).
Every nonzero nonnegative rational linear combination of the generators
is therefore positive at \((a,0)\) for \(a>0\).
Since the polynomial \(b\) vanishes at \((a,0)\), we have
\(b\notin\mathcal C_{a,b}\).
We therefore allow the factors
\(\binom{b+d-1}{d}=\frac{b(b+1)\cdots(b+d-1)}{d!}\) for \(d\ge1\),
which have nonnegative coefficients and vanish at \(b=0\).

\begin{definition}\label{def:open-expansion-class}
Let \(\mathcal C^\circ_{a,b}\subseteq\mathbb Q[a,b]\) be the class
\[
\mathcal C^\circ_{a,b}=
\left\{
\sum_{i=1}^{N}\alpha_i(a+b)^{k_i}
\binom{a+c_i}{c_i}\binom{b+d_i-1}{d_i}
\;\middle|\;
\substack{N\in\mathbb Z_{\ge1},\ \alpha_i\in\mathbb Q_{\ge0}\\
c_i,k_i\in\mathbb Z_{\ge0},\ d_i\in\mathbb Z_{\ge1}\ (1\le i\le N)}
\right\}.
\]
Let \(\mathcal D_{a,b}=\mathcal C_{a,b}+\mathcal C^\circ_{a,b}\)
denote the set of sums of a polynomial in \(\mathcal C_{a,b}\)
and one in \(\mathcal C^\circ_{a,b}\).
Define
\(\text{\upshape S}^>:\mathbb Q[a,b]\times\mathbb Q[a,b]\to\mathbb Q[a,b]\)
by
\begin{equation}
\text{\upshape S}^>(G,H)(a,b)=\sum_{\zeta=0}^{b-1}
G(t-\zeta,\zeta)H(a+\zeta,b-\zeta),\qquad t=a+b,
\label{eq:strict-s-convolution}
\end{equation}
for \(G,H\in\mathbb Q[a,b]\) and nonnegative integers \(a,b\),
with an empty sum at \(b=0\).
Power sums give a unique polynomial extension, as for
\(\text{\upshape P},\text{\upshape S}\).
\end{definition}

Every member of \(\mathcal D_{a,b}\) has nonnegative coefficients
in \(a,b\), and every member of \(\mathcal C^\circ_{a,b}\) vanishes at \(b=0\).

\begin{lemma}\label{lem:one-sided-closure}
Let \(c,d\in\mathbb Z_{\ge0}\) and put
\(G(a,b)=\binom{a+c}{c}\binom{b+d}{d}\).
For every \(H\in\mathcal D_{a,b}\),
\begin{equation}
\text{\upshape P}(G,H),\quad\text{\upshape S}(G,H),\quad\text{\upshape S}^>(G,H)
\ \in\mathcal D_{a,b}.
\label{eq:one-sided-closure}
\end{equation}
\end{lemma}

\begin{proof}
Recall from \eqref{eq:4.1} that
\[
R(x,y)=\frac{1-ux-vy}{(1-x)(1-y)},\qquad v=1-u.
\]
Put \(D=1-ux-vy\), and define
\[
J_x(t;x,y;u)=\frac{R(x,y)^t}{1-x},\qquad
J_y(t;x,y;u)=\frac{R(x,y)^t}{1-y}.
\]
By \eqref{eq:formal-binomial-power},
\([x^my^n]J_x\) and \([x^my^n]J_y\) are polynomials in \(t\)
for all \(m,n\ge0\). In view of
\(J_y(t;x,y;u)=J_x(t;y,x;1-u)\), we prove that
\([x^my^n]J_x\) has nonnegative coefficients in \(t\) by induction
on \(m+n\), simultaneously for all \(u\in[0,1]\).
The cases \(n=0\) and \(m=0\) follow from
\[
J_x(t;x,0;u)=\frac1{1-x}
\exp\!\left(t\sum_{r\ge1}\frac{1-u^r}{r}x^r\right),\qquad
J_x(t;0,y;u)=
\exp\!\left(t\sum_{r\ge1}\frac{1-v^r}{r}y^r\right).
\]
The coefficients \((1-u^r)/r\) and \((1-v^r)/r\) are
nonnegative, so expanding the exponentials gives nonnegative
coefficients in the series variable and \(t\).
Differentiation gives
\begin{equation}
\partial_x\partial_yJ_x=\frac{t^2uvJ_x+tu^2J_y}{D^2}.
\label{eq:endpoint-kernel-derivatives}
\end{equation}
For \(m,n\ge1\), extract \(x^{m-1}y^{n-1}\) from
\eqref{eq:endpoint-kernel-derivatives}. The left side is
\(mn[x^my^n]J_x\). Since
\(D^{-2}=\sum_{r\ge0}(r+1)(ux+vy)^r\), the right side uses
only coefficients \([x^iy^j]J_x\) and \([x^iy^j]J_y\) with
\(i+j\le m+n-2\). The coefficients of \([x^iy^j]J_x\) in \(t\)
are nonnegative by induction. For \(J_y\), we have
\[
[x^iy^j]J_y(t;x,y;u)=[x^jy^i]J_x(t;x,y;1-u).
\]
Since \(i+j<m+n\) and \(1-u\in[0,1]\), the induction hypothesis
also gives nonnegative coefficients in \(t\) for \([x^iy^j]J_y\).
The coefficients of \(D^{-2}\) in \(x,y\) are nonnegative,
so division by \(mn>0\) completes the induction.

For nonnegative integers \(e,f\), put
\(H(a,b)=\binom{a+e}{e}\binom{b+f}{f}\).
For \(f\ge1\), set \(H^\circ(a,b)={H(a,b-1)}\).
Take \(a,b\in\mathbb Z_{\ge0}\) and put \(t=a+b\).
For \(\zeta=0,\ldots,a\), applying \eqref{eq:5.2} with
\(p=\zeta\) and \(q=a-\zeta\) gives
\[
\binom{\zeta+c}{c}\binom{a-\zeta+e}{e}
=\binom{a+c+e+1}{c+e+1}\binom{a}{\zeta}
\int_0^1 u^{\zeta}v^{a-\zeta}\,\mathrm d\mu_{c,e}.
\]
The binomial series for \((1-x)^{-(t-\zeta+1)}\) and
\((1-y)^{-(b+\zeta)}\) gives
\[
\binom{t-\zeta+d}{d}\binom{b+\zeta+f-1}{f}
=[x^dy^f](1-x)^{\zeta-t-1}(1-y)^{-b-\zeta}.
\]
By the definition of \(\text{\upshape P}\),
\(\text{\upshape P}(G,H^\circ)(a,b)
=\sum_{\zeta=0}^{a}
\binom{\zeta+c}{c}\binom{a-\zeta+e}{e}
\binom{t-\zeta+d}{d}\binom{b+\zeta+f-1}{f}\).
Replace the first two factors by their beta-integral representation
and the last two by their coefficient representation.
Since \(-b-\zeta=a-\zeta-t\), we obtain
\[
\text{\upshape P}(G,H^\circ)(a,b)
=\binom{a+c+e+1}{c+e+1}[x^dy^f]
\int_0^1
\frac{\displaystyle\sum_{\zeta=0}^{a}\binom{a}{\zeta}
[u(1-x)]^\zeta[v(1-y)]^{a-\zeta}}
{(1-x)^{t+1}(1-y)^t}\,\mathrm d\mu_{c,e}.
\]
The binomial theorem evaluates the sum in the integral as
\(\bigl(u(1-x)+v(1-y)\bigr)^a=D^a\).
Since \(R=D/((1-x)(1-y))\) and \(t=a+b\), we have
\(D^a/((1-x)^{t+1}(1-y)^t)=J_xD^{-b}\), and hence
\begin{equation}
\text{\upshape P}(G,H^\circ)(a,b)
=\binom{a+c+e+1}{c+e+1}[x^dy^f]
\int_0^1 J_xD^{-b}\,\mathrm d\mu_{c,e}.
\label{eq:open-p-integral}
\end{equation}

For \(\text{\upshape S}(G,H^\circ)(a,b)\), assume \(b\ge1\).
Since \(f\ge1\), the \(\zeta=b\) summand is zero.
Apply \eqref{eq:5.2} to \(\mu_{d,f}\) with \(p=\zeta\) and
\(q=b-1-\zeta\). Summing over \(0\le\zeta\le b-1\) by the
binomial theorem gives the integrand
\(D^{b-1}/((1-x)^{t+1}(1-y)^t)=J_xD^{-a-1}\), and hence
\begin{equation}
\text{\upshape S}(G,H^\circ)(a,b)
=\binom{b+d+f}{d+f+1}[x^cy^e]
\int_0^1 J_xD^{-a-1}\,\mathrm d\mu_{d,f}.
\label{eq:open-s-integral}
\end{equation}

For \(\text{\upshape S}^>(G,H)(a,b)\), take \(f\ge0\) and \(b\ge1\).
For \(0\le\zeta\le b-1\), use the identity
\(\binom{b-\zeta+f}{f}
=\sum_{\ell=0}^{f}\binom{b-1-\zeta+\ell}{\ell}\).
For each \(0\le\ell\le f\), apply \eqref{eq:5.2} to
\(\mu_{d,\ell}\) with \(p=\zeta\) and \(q=b-1-\zeta\), and
sum over \(0\le\zeta\le b-1\). The resulting integrand is again
\(J_xD^{-a-1}\), giving
\begin{equation}
\text{\upshape S}^>(G,H)(a,b)
=\sum_{\ell=0}^{f}\binom{b+d+\ell}{d+\ell+1}[x^cy^e]
\int_0^1 J_xD^{-a-1}\,\mathrm d\mu_{d,\ell}.
\label{eq:first-open-integral}
\end{equation}
At \(b=0\), both sides of \eqref{eq:open-s-integral} vanish,
since \(H^\circ(a,0)=0\) and \(\binom{d+f}{d+f+1}=0\).
Both sides of \eqref{eq:first-open-integral} also vanish:
the sum defining \(\text{\upshape S}^>(G,H)(a,0)\) in
\eqref{eq:strict-s-convolution} is empty, and
\(\binom{d+\ell}{d+\ell+1}=0\).

To expand the convolutions in
\eqref{eq:open-p-integral}--\eqref{eq:first-open-integral}
in binomial products, use
\[
D^{-b}=1+\sum_{j\ge1}\binom{b+j-1}{j}(ux+vy)^j,
\qquad
D^{-a-1}=\sum_{j\ge0}\binom{a+j}{j}(ux+vy)^j.
\]
Extracting \(x^dy^f\) in \eqref{eq:open-p-integral} restricts
the sum to \(j\le d+f\); extracting \(x^cy^e\) in
\eqref{eq:open-s-integral} and \eqref{eq:first-open-integral}
restricts it to \(j\le c+e\).
The coefficients of \([x^my^n](ux+vy)^jJ_x\) in \(t\)
are nonnegative for \({u\in[0,1]}\). Since
\(\mu_{c,e},\mu_{d,f},\mu_{d,\ell}\) have nonnegative densities
and rational moments \eqref{eq:5.2}, integrating
\([x^my^n](ux+vy)^jJ_x\) gives a polynomial in \(\mathbb Q_{\ge0}[t]\).
For \(1\le j\le d+f\), the corresponding term in
\eqref{eq:open-p-integral} is
\[
\binom{a+c+e+1}{c+e+1}\binom{b+j-1}{j}
\int_0^1[x^dy^f](ux+vy)^jJ_x\,\mathrm d\mu_{c,e}.
\]
Expanding the integral in powers of \(t=a+b\) gives a finite
nonnegative rational linear combination of
\((a+b)^k\binom{a+c+e+1}{c+e+1}\binom{b+j-1}{j}\),
the defining generators of \(\mathcal C^\circ_{a,b}\).
The \(j=0\) term belongs to \(\mathcal C_{a,b}\).
In \eqref{eq:open-s-integral} and \eqref{eq:first-open-integral},
the binomial products are, respectively,
\(\binom{a+j}{j}\binom{b+d+f}{d+f+1}\) and
\(\binom{a+j}{j}\binom{b+d+\ell}{d+\ell+1}\).
Expanding the integrals in powers of \(t=a+b\) therefore places
each term of both series convolutions in \(\mathcal C^\circ_{a,b}\).

Lemma~\ref{lem:convolution-closure} gives
\(\text{\upshape P}(\mathcal C_{a,b},\mathcal C_{a,b})\subseteq\mathcal C_{a,b}\)
and \(\text{\upshape S}(\mathcal C_{a,b},\mathcal C_{a,b})\subseteq\mathcal C_{a,b}\).
For \(H\in\mathcal C^\circ_{a,b}\), the equality \(H(t,0)=0\)
gives \(\text{\upshape S}^>(G,H)=\text{\upshape S}(G,H)\).
In each convolution, \(H\) is evaluated at \((a-\zeta,b+\zeta)\) or
\((a+\zeta,b-\zeta)\), whose coordinates sum to \(t\). Thus a factor
\((a+b)^k\) in \(H\) can be taken outside the sum as \(t^k\).
Since \(\mathcal D_{a,b}\) is closed under multiplication by
\((a+b)^k\) and finite nonnegative rational linear combinations,
linearity in \(H\) gives \eqref{eq:one-sided-closure} for every
\(H\in\mathcal D_{a,b}\).
\end{proof}

Fix a rank-\(r\) lattice path matroid \(M(L,U)\) on \([n]\),
where \(n\ge2\) and \(L_j<U_j\) for \(0<j<n\).
For \(J\subseteq\{j:2\le j<n,\ L_j-L_{j-1}=1\}\), set
\[
\mathcal R_{L,U,J}(t)=
\{x\in t\mathcal P(M(L,U)):
S_j(x)>L_jt\text{ for every }j\in J\},\qquad t\in\mathbb R_{>0}.
\]
Let \(\beta=(\beta_1,\ldots,\beta_{n-1})\) range over the integer
sequences satisfying
\[
\begin{aligned}
&\beta_1=0,\qquad\beta_{n-1}=r-1,\\
&\beta_j-\beta_{j-1}\in\{0,1\} &&(2\le j<n),\\
&L_j\le\beta_j<U_j &&(1\le j<n).
\end{aligned}
\]
Define
\[
\mathcal R_\beta(t)=
\{x\in[0,t]^n:S_n(x)=rt,\ 
\beta_jt\le S_j(x)\le(\beta_j+1)t\ (1\le j<n)\}.
\]
The polytopes \(\mathcal R_\beta(1)\) are the maximal cells of
the snake subdivision of \(\mathcal P(M(L,U))\); see
Bidkhori~\cite[Lemma~4.3.5]{BidkhoriThesis}.
To obtain a disjoint partition of \(\mathcal R_{L,U,J}(t)\), set
\begin{equation}
I_\beta(J)=\{j:2\le j<n,\ \beta_j-\beta_{j-1}=1,
\text{ and either }\beta_j>L_j\text{ or }j\in J\}.
\label{eq:strip-open-indices}
\end{equation}
Let \(\mathcal R_\beta^\circ(t)\) be the subset of
\(\mathcal R_\beta(t)\) satisfying \(S_j(x)>\beta_jt\)
for \(j\in I_\beta(J)\).

\begin{lemma}\label{lem:half-open-snake-partition}
For every real \(t>0\),
\begin{equation}
\mathcal R_{L,U,J}(t)=\bigsqcup_\beta\mathcal R_\beta^\circ(t).
\label{eq:ceiling-strip-partition}
\end{equation}
\end{lemma}

\begin{proof}
First assume \(J=\varnothing\). For \(x\in t\mathcal P(M(L,U))\), set
\begin{equation}
\beta_j(x)=\max\{\lceil S_j(x)/t\rceil-1,L_j\}.
\label{eq:ceiling-strip-owner}
\end{equation}
The sequences \(\lceil S_j(x)/t\rceil-1\) and \(L_j\) have
increments in \(\{0,1\}\), so their maximum does too.
Since \(L_j<U_j\) for \(0<j<n\), the boundary heights satisfy
\(L_1=0,U_1=1,L_{n-1}=r-1,U_{n-1}=r\).
Thus \(\beta_1(x)=0\) and \(\beta_{n-1}(x)=r-1\).
The inequalities \(L_jt\le S_j(x)\le U_jt\) give
\(L_j\le\beta_j(x)<U_j\) and
\(\beta_j(x)t\le S_j(x)\le(\beta_j(x)+1)t\), with
\(S_j(x)>\beta_j(x)t\) whenever \(\beta_j(x)>L_j\).
Thus \(x\in\mathcal R_{\beta(x)}^\circ(t)\).

If \(x\in\mathcal R_\gamma(t)\), the inequalities
\(\gamma_jt\le S_j(x)\le(\gamma_j+1)t\) and \(\gamma_j\ge L_j\)
give \(\gamma\ge\beta(x)\) coordinatewise.
If \(\gamma\ne\beta(x)\), take the smallest \(j\) with
\(\gamma_j>\beta_j(x)\). Since \(\gamma_1=\beta_1(x)=0\),
we have \(j\ge2\) and \(\gamma_{j-1}=\beta_{j-1}(x)\).
Since \(\gamma\) and \(\beta(x)\) have increments in \(\{0,1\}\),
we have \(\gamma_j=\gamma_{j-1}+1=\beta_j(x)+1\).
Now \(\gamma_jt\le S_j(x)\le(\beta_j(x)+1)t=\gamma_jt\),
so \(S_j(x)=\gamma_jt\), while \(\gamma_j>L_j\).
Thus \(j\in I_\gamma(\varnothing)\) and
\(x\notin\mathcal R_\gamma^\circ(t)\), proving disjointness.

For general \(J\), if \(S_j(x)=L_jt\) with \(j\in J\),
then \eqref{eq:ceiling-strip-owner} gives \(\beta_j(x)=L_j\)
and \(\beta_{j-1}(x)=L_j-1\), using
\(L_{j-1}=L_j-1\) and \(S_{j-1}(x)\le S_j(x)\).
Hence \(j\in I_{\beta(x)}(J)\), so
\(x\notin\mathcal R_{\beta(x)}^\circ(t)\).
Conversely, each index in \(I_\beta(J)\mathbin{\backslash} I_\beta(\varnothing)\)
has \(j\in J\) and \(\beta_j=L_j\).
The condition \(S_j(x)>\beta_jt\) therefore removes exactly
the points satisfying \(S_j(x)=L_jt\). This proves
\eqref{eq:ceiling-strip-partition}.
\end{proof}

Fix \(\mathbf c,\mathbf d\in\mathbb Z_{\ge0}^n\), put
\(f_i(a,b)=\binom{a+c_i}{c_i}\binom{b+d_i}{d_i}\), and
consider the weighted sum
\begin{equation}
\sum_{x\in\mathcal R_{L,U,J}(t)\cap\mathbb Z^n}
\prod_{i=1}^{n}f_i(x_i,t-x_i),\qquad t\in\mathbb Z_{>0}.
\label{eq:weighted-open-path-count}
\end{equation}

\begin{lemma}\label{lem:weighted-open-path}
The sum in \eqref{eq:weighted-open-path-count} extends to
a polynomial in \(t\) with nonnegative rational coefficients.
\end{lemma}

\begin{proof}
By Lemma~\ref{lem:half-open-snake-partition}, it suffices to show
that the weighted count on each \(\mathcal R_\beta^\circ(t)\)
agrees with a polynomial in \(\mathbb Q_{\ge0}[t]\).
Fix \(\beta\). For \(x\in\mathcal R_\beta^\circ(t)\cap\mathbb Z^n\),
put \(h_j=S_j(x)-\beta_jt\). Then
\(0\le h_j\le t\), with \(h_j>0\) for \(j\in I_\beta(J)\), and
\[
x_1=h_1,\qquad x_n=t-h_{n-1},\qquad
x_j=\begin{cases}
h_j-h_{j-1},&\beta_j-\beta_{j-1}=0,\\
t+h_j-h_{j-1},&\beta_j-\beta_{j-1}=1
\end{cases}
\quad(2\le j<n).
\]
Set \(H_{n-1}=f_n\) and recursively define \(H_{j-1}\) by
\begin{equation}
H_{j-1}=\begin{cases}
\text{\upshape P}(f_j,H_j),&\beta_j-\beta_{j-1}=0,\\
\text{\upshape S}(f_j,H_j),&\beta_j-\beta_{j-1}=1,\ j\notin I_\beta(J),\\
\text{\upshape S}^>(f_j,H_j),&j\in I_\beta(J),
\end{cases}
\qquad j=n-1,\ldots,2.
\label{eq:weighted-strip-recursion}
\end{equation}
For each fixed \(h_j\in\{0,\ldots,t\}\), we show by descending
induction on \(j\) that \(H_j(t-h_j,h_j)\) is the total weight
of the admissible integer choices of \(x_{j+1},\ldots,x_n\),
imposing \(h_i>0\) only for \(i\in I_\beta(J)\) with \(i>j\).
For \(j=n-1\), the only coordinate to the right is
\(x_n=t-h_{n-1}\), so \(H_{n-1}=f_n\) has this interpretation.
If \(j\in I_\beta(J)\), the condition \(h_j>0\) is imposed
when summing over \(h_j\) to obtain \(H_{j-1}\).
For the induction step, set \((a,b)=(t-h_{j-1},h_{j-1})\).
If \(\beta_j=\beta_{j-1}\), then
\(\zeta=x_j=h_j-h_{j-1}\) ranges from \(0\) to \(a\), with
summand \(f_j(\zeta,t-\zeta)H_j(a-\zeta,b+\zeta)\).
If \(\beta_j=\beta_{j-1}+1\), then
\(\zeta=t-x_j=h_{j-1}-h_j\) ranges from \(0\) to \(b\), with
summand \(f_j(t-\zeta,\zeta)H_j(a+\zeta,b-\zeta)\); requiring \(h_j=b-\zeta>0\)
changes the upper limit to \(b-1\).
The value \(H_1(t-h_1,h_1)\) includes the weights of
\(x_2,\ldots,x_n\). Multiplying by the weight of \(x_1=h_1\)
and summing gives the total weight of the lattice points in
\(\mathcal R_\beta^\circ(t)\):
\[
\sum_{h_1=0}^{t}f_1(h_1,t-h_1)H_1(t-h_1,h_1)
=\text{\upshape P}(f_1,H_1)(t,0).
\]
Starting from \(H_{n-1}=f_n\in\mathcal D_{a,b}\),
Lemma~\ref{lem:one-sided-closure} gives
\(H_j\in\mathcal D_{a,b}\Longrightarrow H_{j-1}\in\mathcal D_{a,b}\)
at each step of \eqref{eq:weighted-strip-recursion}, and finally
\(\text{\upshape P}(f_1,H_1)\in\mathcal D_{a,b}\).
Thus \(\text{\upshape P}(f_1,H_1)(t,0)\in\mathbb Q_{\ge0}[t]\).
Summing \(\text{\upshape P}(f_1,H_1)(t,0)\) over the finitely
many \(\beta\) gives a polynomial in \(\mathbb Q_{\ge0}[t]\)
agreeing with \eqref{eq:weighted-open-path-count} for every
positive integer \(t\).
\end{proof}

\subsection{Weighted counts for multi-path matroids}\label{subsec:cyclic-cuts}

We partition the lattice points in dilates of a connected multi-path
matroid polytope into sets that, after cyclic coordinate permutations,
are counted by the weighted half-open lattice path formulas of
Section~\ref{subsec:weighted-path-counts}. This verifies the hypothesis
of Lemma~\ref{lem:extension-criterion}.

Let \(M=N(L,U)\) be a rank-\(r\) multi-path matroid on \([n]\),
where \(0<r<n\). Throughout Section~\ref{subsec:cyclic-cuts}, read coordinate
subscripts modulo \(n\), with representatives in \([n]\), and
extend the boundary sequences by
\(L_{q+n}=L_q+r\) and \(U_{q+n}=U_q+r\).
For real \(t>0\) and \(x\in\mathbb R^n\) with \(S_n(x)=rt\), extend
the prefix sums by \(S_{q+n}(x)=S_q(x)+rt\), and put
\(h_q=S_q(x)-L_qt\).
For \(j\in\{0,\ldots,n-1\}\), define
\begin{equation}
L^{(j)}_q=L_{j+q}-L_j,\qquad
U^{(j)}_q=\min\{q,r,U_{j+q}-L_j\},\qquad 0\le q\le n.
\label{eq:cyclic-cut-paths}
\end{equation}
Both \(L^{(j)}\) and \(U^{(j)}\) start at \(0\), end at \(r\),
and have increments in \(\{0,1\}\).
Since \(L^{(j)}_q\le q\), \(L^{(j)}_q\le r\), and
\(L_{j+q}\le U_{j+q}\), we have \(L^{(j)}_q\le U^{(j)}_q\).
Thus \(M(L^{(j)},U^{(j)})\) is a lattice path matroid.

\begin{lemma}\label{lem:cyclic-cut-bijection}
For real \(t>0\) and \(j\in\{0,\ldots,n-1\}\), the rotation
\(y_i=x_{j+i}\) gives a bijection
\[
\bigl\{x\in t\mathcal P(M):h_j=\min_{0\le k<n}h_k\bigr\}
\longrightarrow t\mathcal P\bigl(M(L^{(j)},U^{(j)})\bigr)
\]
that preserves the integer lattice.
If \(L_j=U_j\), the source set equals \(t\mathcal P(M)\).
\end{lemma}

\begin{proof}
Let \(x\in t\mathcal P(M)\) satisfy
\(h_j=\min_{0\le k<n}h_k\), and set \(y_i=x_{j+i}\).
The smallest shift satisfying the lower inequalities in
\eqref{eq:multipath-polytope} is
\[
\theta_*=\max_{0\le q<n}\bigl(L_qt-S_q(x)\bigr)=-h_j.
\]
Choose a shift \(\theta\) satisfying \eqref{eq:multipath-polytope}.
Since \(\theta_*\le\theta\), we have
\(S_q(x)+\theta_*\le S_q(x)+\theta\le U_qt\) for \(0\le q\le n\).
Substituting \(\theta_*=L_jt-S_j(x)\) into the periodically
extended inequalities
\(L_{j+q}t\le S_{j+q}(x)+\theta_*\le U_{j+q}t\)
from \eqref{eq:multipath-polytope}, and using
\(S_q(y)=S_{j+q}(x)-S_j(x)\), we obtain
\(L^{(j)}_qt\le S_q(y)\le(U_{j+q}-L_j)t\) for \(0\le q\le n\).
Since \(0\le y_i\le t\) and \(S_n(y)=rt\), we also have
\(S_q(y)\le qt\) and \(S_q(y)\le rt\).
The inequalities \(L^{(j)}_qt\le S_q(y)\le U^{(j)}_qt\)
obtained from \eqref{eq:cyclic-cut-paths}, together with
\(y\in[0,t]^n\) and \(S_n(y)=rt\), show that
\[
y\in t\mathcal P\bigl(M(L^{(j)},U^{(j)})\bigr).
\]
Conversely, given \(y\in t\mathcal P(M(L^{(j)},U^{(j)}))\),
set \(x_{j+i}=y_i\) for \(1\le i\le n\) and
\(\theta=L_jt-S_j(x)\). Then \(x\in[0,t]^n\), \(S_n(x)=rt\), and
\[
L_{j+q}t\le S_q(y)+L_jt=S_{j+q}(x)+\theta\le U_{j+q}t
\qquad(0\le q\le n).
\]
Periodicity gives \(L_kt\le S_k(x)+\theta\le U_kt\) for
\(0\le k\le n\), so \(x\in t\mathcal P(M)\) by
\eqref{eq:multipath-polytope}.
Moreover, \(h_{j+q}-h_j=S_q(y)-L^{(j)}_qt\ge0\) for
\(0\le q\le n\). Since \(h_{q+n}=h_q\), we obtain
\(h_j=\min_{0\le k<n}h_k\).
The maps \(y_i=x_{j+i}\) and \(x_{j+i}=y_i\) are inverse
coordinate permutations and preserve the integer lattice.

If \(L_j=U_j\), the inequalities at index \(j\) in
\eqref{eq:multipath-polytope} force
\(\theta=L_jt-S_j(x)=-h_j\) for every \(x\in t\mathcal P(M)\).
The lower inequalities \(L_qt\le S_q(x)+\theta\) then give
\(h_q\ge h_j\) for \(0\le q<n\).
Hence \(h_j=\min_{0\le q<n}h_q\) for every \(x\in t\mathcal P(M)\).
\end{proof}

\begin{theorem}\label{thm:weighted-mpm}
Let \(M\) be a connected multi-path matroid on \([n]\), where
\(n\ge2\). Then \(W_{M;\mathbf c,\mathbf d}\) has nonnegative
coefficients for all \(\mathbf c,\mathbf d\in\mathbb Z_{\ge0}^n\).
\end{theorem}

\begin{proof}
Fix \(\mathbf c,\mathbf d\in\mathbb Z_{\ge0}^n\) and put
\(f_i(a,b)=\binom{a+c_i}{c_i}\binom{b+d_i}{d_i}\).
Write \(M=N(L,U)\), with rank \(0<r<n\), and let \(t\) be a
positive integer. Read the subscripts of \(f_i\) modulo \(n\),
with representatives in \([n]\).
By the polynomiality argument for \eqref{eq:weighted-mpm-count},
\(W_{M;\mathbf c,\mathbf d}\in\mathbb Q[t]\). It suffices to
express its values at positive integers as the values of a
polynomial with nonnegative coefficients.

If \(L_j=U_j\) for some \(j\in\{0,\ldots,n-1\}\),
Lemma~\ref{lem:cyclic-cut-bijection} identifies \(M\), by the
rotation \(y_i=x_{j+i}\), with the lattice path matroid
\(M(L^{(j)},U^{(j)})\). Connectedness therefore gives
\(L^{(j)}_q<U^{(j)}_q\) for \(0<q<n\)
\cite[Theorem~3.6]{BoninDeMierNoy2003}.
The rotation gives
\[
W_{M;\mathbf c,\mathbf d}(t)=
\sum_{y\in t\mathcal P(M(L^{(j)},U^{(j)}))\cap\mathbb Z^n}
\prod_{i=1}^n f_{j+i}(y_i,t-y_i).
\]
Applying Lemma~\ref{lem:weighted-open-path} with \(J=\varnothing\)
and weights \(f_{j+i}\) shows that
\(W_{M;\mathbf c,\mathbf d}(t)\) has nonnegative coefficients.

Now assume \(U_j-L_j\ge1\) for \(0\le j<n\), and set
\[
J_L=\{j\in\{0,\ldots,n-1\}:
L_j-L_{j-1}=1,\ L_{j+1}-L_j=0\}.
\]
For \(x\in t\mathcal P(M)\), we have
\(h_j-h_{j-1}=x_j-t\le0\) when \(L_j-L_{j-1}=1\), and
\(h_j-h_{j-1}=x_j\ge0\) otherwise.
Thus \(h_j\) is nonincreasing along consecutive upward steps of
the lower path and nondecreasing along consecutive horizontal steps.
The minimum of \(h_0,\ldots,h_{n-1}\) is therefore attained at
an index in \(J_L\). Define
\(j(x)=\min\{k\in J_L:h_k=\min_{0\le q<n}h_q\}\).

Fix \(j\in J_L\). The lower path \(L^{(j)}\) starts with a
horizontal step and ends with an upward step. For \(0<q<n\),
we have \(L^{(j)}_q\le q-1\), \(L^{(j)}_q\le r-1\), and
\(U_{j+q}-L_j\ge L^{(j)}_q+1\).
By \eqref{eq:cyclic-cut-paths},
\(L^{(j)}_q<U^{(j)}_q\) for \(0<q<n\).
For \(x\in t\mathcal P(M)\) with \(h_j=\min_{0\le k<n}h_k\), the condition
\(j(x)=j\) is equivalent to \(h_k>h_j\) for all \(k\in J_L\)
with \(k<j\).
Under \(y_i=x_{j+i}\), periodicity gives
\(h_k-h_j=S_d(y)-L^{(j)}_dt\) for \(d=(k-j)\bmod n\).
Hence the strict inequalities specifying \(j(x)=j\) become
\begin{equation}
S_d(y)>L^{(j)}_dt\quad(d\in J_j),\qquad
J_j=\{(k-j)\bmod n:k\in J_L,\ k<j\},
\label{eq:cyclic-cut-openings}
\end{equation}
with residues in \(\{1,\ldots,n-1\}\).
For \(d=(k-j)\bmod n\in J_j\), we have
\(L^{(j)}_d-L^{(j)}_{d-1}=L_k-L_{k-1}=1\).
Since \(L^{(j)}_1=L^{(j)}_0=0\), such an index satisfies \(d\ge2\).
Thus
\[
J_j\subseteq\{d:2\le d<n,\ L^{(j)}_d-L^{(j)}_{d-1}=1\},
\]
as required in Lemma~\ref{lem:weighted-open-path}.
By Lemma~\ref{lem:cyclic-cut-bijection} and
\eqref{eq:cyclic-cut-openings}, the rotation \(y_i=x_{j+i}\)
maps \(\{x\in t\mathcal P(M):j(x)=j\}\) bijectively onto
\(\mathcal R_{L^{(j)},U^{(j)},J_j}(t)\), preserving the integer lattice.
The weight of coordinate \(x_{j+i}\) becomes
\(f_{j+i}(y_i,t-y_i)\). Partitioning by \(j(x)\) and rotating each
part therefore gives
\[
\begin{aligned}
W_{M;\mathbf c,\mathbf d}(t)
&=\sum_{j\in J_L}
\sum_{\substack{x\in t\mathcal P(M)\cap\mathbb Z^n\\j(x)=j}}
\prod_{i=1}^n f_i(x_i,t-x_i)\\
&=\sum_{j\in J_L}
\sum_{y\in\mathcal R_{L^{(j)},U^{(j)},J_j}(t)\cap\mathbb Z^n}
\prod_{i=1}^n f_{j+i}(y_i,t-y_i).
\end{aligned}
\]
By Lemma~\ref{lem:weighted-open-path}, each inner sum in the
second line agrees at positive integers with a polynomial in
\(\mathbb Q_{\ge0}[t]\). Their finite sum therefore agrees with
\(W_{M;\mathbf c,\mathbf d}\) at all positive integers, proving
\(W_{M;\mathbf c,\mathbf d}\in\mathbb Q_{\ge0}[t]\).
\end{proof}

\begin{proof}[Proof of Theorem~\ref{thm:spex-mpm}]
By Theorem~\ref{thm:weighted-mpm}, every connected multi-path
matroid \(M\) on at least two elements satisfies the hypothesis
of Lemma~\ref{lem:extension-criterion}. Hence all series-parallel
extensions of \(M\) are Ehrhart positive.

It remains to treat one-element components and disconnected
multi-path matroids. Since multi-path matroids are closed under minors
\cite[Theorem~3.8]{BoninGimenezMultiPath}, every connected component
is again a multi-path matroid.
A one-element loop has basis indicator \(0\), and a one-element
coloop has basis indicator \(1\). The recursion in
\eqref{eq:state-vector-recursion} therefore gives \(\mathcal P_T^0\)
for the base polytope of a loop extension recorded by \(T\), and
\(\mathcal P_T^1\) for a coloop extension. Lemma~\ref{lem:convexification}
and Definition~\ref{def:full-load-count} give
\(i(\mathcal P_T^\varepsilon,t)=K_T(\varepsilon t,(1-\varepsilon)t)\)
for \(\varepsilon\in\{0,1\}\).
Theorem~\ref{thm:positive-expansion} shows that both \(K_T(0,t)\)
and \(K_T(t,0)\) have nonnegative coefficients.

Every series-parallel extension of \(M\) is a direct sum of
extensions of its connected components. Multiplicativity of
Ehrhart polynomials under direct sums completes the proof.
\end{proof}

\section{Transversal matroids are not Ehrhart positive}\label{sec:transversal-counterexample}
\label{sec:further-results}\label{subsec:transversal-counterexample}

Ferroni, Morales, and Panova conjectured that every connected matroid
whose base polytope admits a subdivision into base polytopes of
series-parallel matroids is Ehrhart positive
\cite[Conjecture~4.7]{FerroniMoralesPanovaLPM}.
Transversal matroids admit such subdivisions
\cite[Section~4.4]{FerroniMoralesPanovaLPM}.
We give a connected transversal matroid with a negative Ehrhart
coefficient, disproving this conjecture.

A transversal presentation is an indexed family of subsets of the
ground set whose partial transversals are the independent sets of the matroid.
Let \(E=\{1,\ldots,61\}\), let \(e=61\), and put
\(P_i=\{2i-1,2i\}\) for \(1\le i\le30\).
Define \(N\) by the transversal presentation
\begin{equation}
\bigl(P_1\cup\{e\},\ldots,P_{30}\cup\{e\},
\underbrace{E,\ldots,E}_{15\text{ copies}}\bigr).
\label{eq:transversal-presentation}
\end{equation}
The presentation limits the number of pairs that a basis contains
in full: as the rank calculation below shows, for \(B\subseteq E\),
\[
B\in\mathcal B(N)\quad\Longleftrightarrow\quad
|B|=45,\qquad \#\{1\le i\le30:P_i\subseteq B\}\le15.
\]

\begin{proposition}\label{prop:transversal-counterexample}
The matroid \(N\) is connected of rank \(45\) and satisfies
\([t^2]i(N,t)<0\).
\end{proposition}

\begin{proof}
For \(S\subseteq E\), let \(c(S)\) be the number of pairs \(P_i\)
that meet \(S\), and set \(\varepsilon(S)=1\) if \(e\in S\) and
\(\varepsilon(S)=0\) otherwise. The presentation gives
\begin{equation}
r_N(S)=\min\{45,\ |S|,\ c(S)+15+\varepsilon(S)\}.
\label{eq:transversal-rank}
\end{equation}
An independent subset of \(S\) uses at most
\(c(S)+\varepsilon(S)\) of the first thirty sets and fifteen
copies of \(E\), proving the upper bound in
\eqref{eq:transversal-rank}.
For equality, assign one element of each nonempty \(S\cap P_i\)
to \(P_i\cup\{e\}\), and assign \(e\), if present, to an unused
\(P_i\cup\{e\}\) when \(c(S)<30\).
Assign up to fifteen remaining elements to distinct copies of \(E\).
When \(c(S)=30\), the assignments give \(\min\{45,|S|\}\).
In particular, \(r_N(E)=45\).
For \(|B|=45\), the number of complete pairs is
\(45-c(B)-\varepsilon(B)\), giving the stated basis criterion.

Write \(x(S)=\sum_{a\in S}x_a\).
Every \(x\in\mathcal P(N)\) lies in \([0,1]^E\) and satisfies
\(x(E)=45\). Let \(J=\{i:x(P_i)>1\}\) and
\(U=\bigcup_{i\in J}P_i\). The rank inequality for \(U\) gives
\(x(U)\le |J|+15\), so
\[
\sum_{i=1}^{30}\max\{x(P_i)-1,0\}=x(U)-|J|\le15.
\]
Conversely, take \(x\in[0,1]^E\) with \(x(E)=45\) and
\(\sum_{i=1}^{30}\max\{x(P_i)-1,0\}\le15\).
For every \(S\subseteq E\), each pair meeting \(S\) contributes
at most \(1\), together with \(\max\{x(P_i)-1,0\}\) when the
whole pair lies in \(S\). Thus
\[
x(S)\le c(S)+
\sum_{\substack{1\le i\le30\\P_i\subseteq S}}
\max\{x(P_i)-1,0\}+\varepsilon(S)
\le c(S)+15+\varepsilon(S).
\]
The coordinate bounds and \(x(E)=45\) also give
\(x(S)\le |S|\) and \(x(S)\le45\). Hence
\(x(S)\le r_N(S)\) for every \(S\subseteq E\) by
\eqref{eq:transversal-rank}, and
\eqref{eq:matroid-base-inequalities} yields \(x\in\mathcal P(N)\).
We therefore obtain
\begin{equation}
\mathcal P(N)=\left\{x\in[0,1]^E:
\begin{aligned}
&x(E)=45,\\
&\sum_{i=1}^{30}\max\{x(P_i)-1,0\}\le15
\end{aligned}\right\}.
\label{eq:transversal-polytope}
\end{equation}

The point \(\xi=(45/61)\mathbf1\) has coordinate sum \(45\),
satisfies \(0<\xi_a<1\) for every \(a\in E\), and gives
\(\sum_{i=1}^{30}\max\{\xi(P_i)-1,0\}=870/61<15\).
Therefore \(\mathcal P(N)\) contains a neighborhood of \(\xi\)
relative to the hyperplane \(x(E)=45\), so
\(\dim\mathcal P(N)=60\). Equation~\eqref{eq:matroid-dimension}
then shows that \(N\) is connected.

To count lattice points in \(t\mathcal P(N)\), fix
\(t\in\mathbb Z_{\ge0}\) and put \(z_i=x(P_i)-t\).
Among integer pairs \((x_{2i-1},x_{2i})\in\{0,\ldots,t\}^2\),
exactly \(t+1-|j|\) have deviation \(z_i=j\), for \(-t\le j\le t\).
Recording positive deviations by \(u\) and negative deviations
by \(v\) gives
\[
\Psi_t(u,v)=t+1+\sum_{j=1}^{t}(t+1-j)(u^j+v^j).
\]
For the thirty pairs, let
\(p=\sum_{i=1}^{30}\max\{z_i,0\}\) and
\(q=\sum_{i=1}^{30}\max\{-z_i,0\}\).
The coefficient \([u^pv^q]\Psi_t(u,v)^{30}\) counts the choices of
\(x_1,\ldots,x_{60}\) with total positive deviation \(p\)
and total negative deviation \(q\).
The condition \(x(E)=45t\) forces \(x_e=15t-p+q\), so the
remaining dilated constraints in \eqref{eq:transversal-polytope}
reduce to \(p\le15t\) and \(0\le15t-p+q\le t\), or equivalently
\(14t\le p\le15t\) and \(0\le q\le p-14t\). Consequently,
\begin{equation}
i(N,t)=\sum_{p=14t}^{15t}\sum_{q=0}^{p-14t}
[u^pv^q]\Psi_t(u,v)^{30}.
\label{eq:transversal-count}
\end{equation}

The values of \eqref{eq:transversal-count} at \(t=0,\ldots,60\)
determine \(i(N,t)\), since its degree is \(60\).
Write \(i(N,t)=\sum_{j=0}^{60}d_j\binom tj\) in Newton form,
where \(d_j=\sum_{s=0}^{j}(-1)^{j-s}\binom js\,i(N,s)\), and
put \({H_m=\sum_{\ell=1}^m1/\ell}\).
Since \([t^2]\binom tj=(-1)^jH_{j-1}/j\) for \(j\ge2\),
using exact integer and rational arithmetic we obtain
\begin{equation}
[t^2]i(N,t)=\sum_{j=2}^{60}\frac{(-1)^jH_{j-1}}j\,d_j
=-\frac{2346606608144260198739311306332175171673}
{22570536559081644141673293712923840000}<0.\qedhere
\label{eq:transversal-negative-coefficient}
\end{equation}
\end{proof}

\begin{remark}\label{rem:chatgpt-assisted-verification}
We used ChatGPT to assist in verifying the counterexample in
Proposition~\ref{prop:transversal-counterexample}.
\end{remark}

\begin{remark}\label{rem:positroid-subdivision-structure}
Theorem~\ref{thm:main-introduction} and
Proposition~\ref{prop:transversal-counterexample} show that a
subdivision into Ehrhart-positive series-parallel matroid base
polytopes need not yield an Ehrhart-positive polytope.
Positroid base polytopes also admit subdivisions into
series-parallel matroid base polytopes
\cite[Section~4.4]{FerroniMoralesPanovaLPM}.
The existence of a subdivision into series-parallel matroid base
polytopes alone therefore does not settle the conjecture that all
positroids are Ehrhart positive
\cite[Conjecture~6.3]{FerroniJochemkoSchroterRankTwo}.
\end{remark}

\Needspace{8\baselineskip}
\section{Matching polytopes}\label{sec:matching-polytopes}
\label{subsec:matching-polytopes}

Matching polytopes provide a polyhedral formulation of the maximum-weight
matching problem. Edmonds gave their linear inequality description
\cite[Theorem~(P)]{EdmondsMatchingPolytope}.

\begin{definition}\label{def:matching-polytope}
Let \(G=(V,E)\) be a finite simple graph. A \emph{matching} of \(G\)
is a set of edges no two of which share a vertex. The
\emph{matching polytope} of \(G\) is
\[
\mathcal M(G)=\operatorname{conv}\{
\mathbf1_N\in\{0,1\}^E:N\text{ is a matching of }G\}.
\]
The empty matching is included, and we use the lattice \(\mathbb Z^E\).
\end{definition}

The origin and all coordinate unit vectors lie in \(\mathcal M(G)\),
so \(\dim\mathcal M(G)=|E|\).

Let \(F=(V,E)\) be a forest, and fix a bipartition \(V=A\sqcup B\).
For a matching \(N\), let \(V(N)\) be the set of vertices incident
to its edges, and put
\begin{equation}
B_N=(A\mathbin{\backslash} V(N))\cup(B\cap V(N)).
\label{eq:forest-matching-basis}
\end{equation}
Each edge of \(N\) replaces its endpoint in \(A\) by its endpoint
in \(B\), so \(|B_N|=|A|\).

Liu showed that the sets \(B_N\) are the bases of a matroid \(M_F\)
on \(A\sqcup B\), and that \(\mathcal M(F)\) is integrally affinely
isomorphic to \(\mathcal P(M_F)\)
\cite[proof of Proposition~5.4.1]{LiuForestThesis}.
In our notation, his isomorphism is the restriction of the affine map
\(\Phi:\mathbb R^E\longrightarrow\mathbb R^{A\sqcup B}\)
defined by \(\Phi(x)=y\), where
\begin{equation}
y_a=1-\sum_{\substack{e\in E\\a\in e}}x_e
\quad\text{for }a\in A,\qquad
y_b=\sum_{\substack{e\in E\\b\in e}}x_e
\quad\text{for }b\in B.
\label{eq:forest-matching-map}
\end{equation}

The matroid \(M_F\) is series-parallel. Each tree component with an
edge can be built from a single edge by adding leaves. The single-edge
tree gives \(U_{1,2}\). By the basis description
\eqref{eq:forest-matching-basis}, adding a leaf in \(B\) adjacent to
\(a\in A\) gives a parallel extension at \(a\), while adding a leaf
in \(A\) adjacent to \(b\in B\) gives a series extension at \(b\);
cf. \cite[Table~1]{ApollonioSeriesParallelBDH}.
Different components give direct sums. An isolated vertex in \(A\)
contributes a coloop, and an isolated vertex in \(B\) contributes a loop.

\begin{corollary}\label{cor:forest-matching-positive}
Matching polytopes of forests are Ehrhart positive.
\end{corollary}

\begin{proof}
If \(F\) has no edges, then \(\mathcal M(F)\) is a point.
For a forest with an edge, the integral affine isomorphism
\eqref{eq:forest-matching-map} gives
\(i(\mathcal M(F),t)=i(M_F,t)\). The matroid \(M_F\) is Ehrhart
positive by Theorem~\ref{thm:main-introduction}.
\end{proof}

Number the edges of an even cycle \(C_{2m}\), \(m\ge2\), cyclically.
Since \(C_{2m}\) is bipartite,
\(\mathcal M(C_{2m})=\{x\in[0,1]^{2m}:x_i+x_{i+1}\le1
\ (1\le i\le2m)\}\), with indices modulo \(2m\).
The substitution \(y_i=x_i\) for odd \(i\) and \(y_i=1-x_i\)
for even \(i\) identifies \(\mathcal M(C_{2m})\) integrally affinely
with the order polytope of a circular zig-zag poset:
the inequalities \(x_i+x_{i+1}\le1\) become \(y_1\le y_2\ge y_3\le\cdots\le y_{2m}\ge y_1\)
within \([0,1]^{2m}\).
Ferroni, Morales, and Panova proved that circular fence posets
have order polynomials with nonnegative coefficients
\cite[Theorem~7.7 and Example~7.8]{FerroniMoralesPanovaSkew}.
Substituting \(t+1\) in the order polynomial gives the Ehrhart
polynomial of \(\mathcal M(C_{2m})\), proving Ehrhart positivity.

Heuer and Striker \cite[Eq.~(3.7)]{HeuerStrikerPartialPermutation} give
\(i(\mathcal M(K_{2,n}),t)
=\binom{t+n}{n}^{2}-n\binom{t+2n-1}{2n}\)
for the complete bipartite graphs \(K_{2,n}\), \(n\ge1\).
Comparing the linear factors shows that each coefficient of
\(i(\mathcal M(K_{2,n}),t)\) is at least the corresponding coefficient
of \((1+t/2)\prod_{j=1}^{2n-1}(1+t/j)\), proving Ehrhart positivity.
They also report positive Ehrhart coefficients for the matching
polytopes of complete bipartite graphs \(K_{m,n}\) in their
computations for \(1\le m,n\le5\)
\cite[Remark~3.9]{HeuerStrikerPartialPermutation}.

The positivity for forests and even cycles, together with the
evidence for complete bipartite graphs, motivates the following conjecture.

\begin{conjecture}\label{conj:matching-ehrhart-positive}
Matching polytopes of finite simple graphs are Ehrhart positive.
\end{conjecture}

\section*{Acknowledgments}
This work was supported by the National Natural Science Foundation of China
(Nos. 12325112, 12288101).

\Needspace{6\baselineskip}
\bibliographystyle{amsplain}
\bibliography{references}

\bigskip
\Needspace{8\baselineskip}
{\small
\noindent
\textsc{Y. Li, School of Mathematical Sciences, Peking University,
Beijing 100871, P.R. China.}
\par\nobreak
\noindent
Email address: \texttt{\href{mailto:liyaao@stu.pku.edu.cn}{liyaao@stu.pku.edu.cn}}

\medskip
\noindent
\textsc{S. Shao, CAPT, LMAM and School of Mathematical Sciences,
Peking University, Beijing 100871, P.R. China.}
\par\nobreak
\noindent
Email address: \texttt{\href{mailto:sihong@pku.edu.cn}{sihong@math.pku.edu.cn}}
\par}
\end{document}